\documentclass{amsart}   	
\usepackage{amssymb,amsthm}
\usepackage{amsthm}
\usepackage{graphicx}
\usepackage{enumitem}
\usepackage{tikz}
\usepackage{tikz-cd}
\usetikzlibrary{calc}
\usetikzlibrary{decorations.pathmorphing}
\usepackage{hyperref}
\usepackage{adjustbox}
\usepackage{mathtools}
\usepackage[all]{xy}

\tikzset{curve/.style={settings={#1},to path={(\tikztostart)
    .. controls ($(\tikztostart)!\pv{pos}!(\tikztotarget)!\pv{height}!270:(\tikztotarget)$)
    and ($(\tikztostart)!1-\pv{pos}!(\tikztotarget)!\pv{height}!270:(\tikztotarget)$)
    .. (\tikztotarget)\tikztonodes}},
    settings/.code={\tikzset{quiver/.cd,#1}
        \def\pv##1{\pgfkeysvalueof{/tikz/quiver/##1}}},
    quiver/.cd,pos/.initial=0.35,height/.initial=0}

\tikzset{tail reversed/.code={\pgfsetarrowsstart{tikzcd to}}}
\tikzset{2tail/.code={\pgfsetarrowsstart{Implies[reversed]}}}
\tikzset{2tail reversed/.code={\pgfsetarrowsstart{Implies}}}
\tikzset{no body/.style={/tikz/dash pattern=on 0 off 1mm}}

\graphicspath{ {./images/} }

\newtheorem{theorem}{Theorem}[section]
\newtheorem{proposition}[theorem]{Proposition}
\newtheorem{lemma}[theorem]{Lemma}

\theoremstyle{definition}
\newtheorem{definition}[theorem]{Definition}
\newtheorem{remark}[theorem]{Remark}
\newtheorem{example}[theorem]{Example}

\newcommand{\qedno}{\null\nobreak\hfill\ensuremath{\square}} 
\newcommand{\uxa}{\ensuremath{(\underline{X},\underline{A})}} 
\newcommand{\ux}{\ensuremath{(X,\ast)}}

\newcommand{\cxx}{\ensuremath{(\underline{CX},\underline{X})}}

\newcommand{\clxx}{\ensuremath{(\underline{C\Omega X},\underline{\Omega X})}}

\newcommand{\djk}{\ensuremath{DJ_{K}}}
\newcommand{\zk}{\ensuremath{\mathcal{Z}_{K}}} 
\newcommand{\conn}{\ensuremath{\#}} 

\begin{document}
\title{Loop space decompositions of highly symmetric spaces with applications to polyhedral products}
\author{Lewis Stanton*}


\begin{abstract} 
We generalise the fold map for the wedge sum and use this to give a loop space decomposition of topological spaces with a high degree of symmetry. This is applied to polyhedral products to give a loop space decomposition of polyhedral products associated to families of graphs.\vspace{-3em}
\end{abstract}
\maketitle
\section{Introduction}
\label{sec:intro} 

The motivation for this paper is to study polyhedral products associated to graphs with a high degree of symmetry. This is the first systematic study of polyhedral products associated to graphs. To do this, in the first part of the paper, we prove a general loop space decomposition theorem that holds for topological spaces constructed iteratively via a generalisation of the fold map. This is used as a tool for the main objective of this paper, which is to prove a loop space decomposition of polyhedral products associated to a particular family of graphs known as generalised book graphs. The generalised book graph case also requires a novel result about the naturality of a homotopy equivalence derived from a certain homotopy pushout considered in \cite{GT1}, which is of interest in its own right.

We first provide the necessary setup to state the general loop space decomposition. Let $X$ be a pointed $CW$-complex. The fold map $F:X \vee X \rightarrow X$ on the wedge sum is defined as a pushout map resulting in a commutative diagram \[\begin{tikzcd}
	{*} & X \\
	X & {X \vee X} \\
	&& X.
	\arrow[from=2-2, to=3-3]
	\arrow[from=1-2, to=2-2]
	\arrow[from=2-1, to=2-2]
	\arrow[from=1-1, to=2-1]
	\arrow[from=1-1, to=1-2]
	\arrow[curve={height=-12pt}, Rightarrow, no head, from=1-2, to=3-3]
	\arrow[curve={height=12pt}, Rightarrow, no head, from=2-1, to=3-3]
\end{tikzcd}\] Intuitively, this takes each copy of $X$ in the wedge and folds them onto a single copy of $X$.

In this paper, we generalise the idea of the fold map and determine its homotopy fibre. In particular, for $n \geq 2$ let $X_1,\cdots,X_n$ be homeomorphic, path-connected CW-complexes, and for $1 \leq i < j \leq n$, let $\phi_{i,j}:X_i \rightarrow X_j$ be a homeomorphism. Let $A_1$, $A_2$ be homeomorphic subcomplexes of $X_1$. Also, assume there is an automorphism $\psi$ of $X_1$ such that $\psi(A_1) = A_2$ and $\psi(A_2) = A_1$. Define the space $P_2$ as the pushout \[\begin{tikzcd}
	{A_2} & {X_2} \\
	{X_1} & {P_2}
	\arrow["{f_2}", from=1-1, to=1-2]
	\arrow[from=1-2, to=2-2]
	\arrow[hook, from=1-1, to=2-1]
	\arrow[from=2-1, to=2-2]
\end{tikzcd}\] where $f_2$ is the composite \[A_2 \xrightarrow{\psi|_{A_2}} A_1 \xrightarrow{\phi_{1,2}|_{A_1}} \phi_{1,2}(A_1) \hookrightarrow X_2.\] This corresponds to gluing $X_{1}$ to $X_2$ by gluing the copy of $A_{2}$ in $X_1$ to the copy of $A_{1}$ in $X_2$. Inductively define $P_n$ as the pushout \[\begin{tikzcd}
	{\phi_{1,n-1}(A_{2})} & {X_n} \\
	{P_{n-1}} & {P_n}
	\arrow[from=2-1, to=2-2]
	\arrow[from=1-2, to=2-2]
	\arrow["{f_n}", from=1-1, to=1-2]
	\arrow[hook, from=1-1, to=2-1]
\end{tikzcd}\] where $f_n$ is the composite \[\phi_{1,n-1}(A_{2}) \xrightarrow{\phi_{1,n-1}^{-1}|_{\phi_{1,n-1}(A_2)}} A_2 \xrightarrow{\psi|_{A_2}} A_1 \xrightarrow{\phi_{1,n}|_{A_1}} \phi_{1,n}(A_1) \hookrightarrow X_n.\] This corresponds to gluing the copy of $X_{n-1}$ contained in $P_{n-1}$ to $X_n$ by gluing the copy of $A_{2}$ in $X_{n-1}$ to the copy of $A_{1}$ in $X_n$. We define a \textit{fold map} \[g_n:P_n \rightarrow X_1\] by sending each copy of $X_i$ via the homeomorphism $\phi_{1,i}^{-1}$ into $X_1$. The wedge sum $X \vee X \rightarrow X$ corresponds to letting $n=2$, $A_1 = A_2 = *$, $X_1 = X_2 = X$, $\phi = id_X$ and $\psi = id_X$, where $id_X: X\rightarrow X$ is the identity map. In Section ~\ref{sec:symsimpdecomp}, we decompose the homotopy fibre of $g_n$ and use this to give a loop space decomposition of $P_n$. Specifically, we prove the following result.

\begin{theorem}
\label{symDecomp}

Let $P_n$ be the topological space constructed as above where $n \geq 2$. Let $F_{A_2}$ be the homotopy fibre of the inclusion $A_2 \hookrightarrow X_1$. Then \[\Omega P_n \simeq \Omega X_1 \times \Omega F\] where $F \simeq \bigvee\limits_{i=1}^{n-1} \Sigma F_{A_2}$.
\end{theorem}

As examples, in Section ~\ref{sec:symsimpdecomp} we apply Theorem ~\ref{symDecomp} to the wedge sum $\bigvee_{i=1}^n X$ to obtain a version of the Hilton-Milnor theorem, and to the connected sum $X \conn X$, when $X$ is a closed, $n$-dimensional manifold, to obtain a homotopy decomposition for $\Omega(X \conn X)$. The decomposition of $\Omega(X \conn X)$ recovers a special case of \cite{BT} using a different approach, and a special case of \cite[Lemma 2.1]{HT2}.

The main motivation behind Theorem ~\ref{symDecomp} in this paper is to study polyhedral products associated to graphs. Polyhedral products have attracted considerable attention due to their diverse range of applications across mathematics. They originated in the study of spaces called moment-angle manifolds, but have since been applied to number theory, combinatorics, free groups and robotics to name just a few areas \cite{BBC}.
  
A polyhedral product is a natural subspace of $\prod_{i=1}^m X_i$ defined as follows. Let $K$ be a simplicial complex on the vertex set $[m] = \{1,2,\cdots,m\}$. For $1 \leq i \leq m$, let $(X_i,A_i)$ be a pair of pointed $CW$-complexes, where $A_i$ is a pointed $CW$-subcomplex of $X_i$. Let $\uxa = \{(X_i,A_i)\}_{i=1}^m$ be the sequence of pairs. For each simplex $\sigma \in K$, let $\uxa^\sigma$ be defined by \[ \uxa^\sigma = \prod\limits_{i=1}^m Y_i \text{ where }  Y_i = \begin{cases} X_i & i \in \sigma \\ A_i & i \notin \sigma. \end{cases}\] The \textit{polyhedral product} determined by $\uxa$ and $K$ is \[\uxa^K = \bigcup\limits_{\sigma \in K} \uxa^{\sigma} \subseteq \prod\limits_{i=1}^m X_i.\]

For example, if $A_i$ is a point for all $i$ and $K$ is $m$ disjoint points, then $(\underline{X},\underline{*})^K = X_1 \vee \cdots \vee X_m$. However, if $K$ is the full $(m-1)$-simplex $\Delta^{m-1}$, then $(\underline{X},\underline{*})^K = \prod_{i=1}^m X_i$. In this paper, we specialise to the case when $X_i=X$ and $A_i=A$ for all $i$, and in this case, the polyhedral product is denoted by $(X,A)^K$. 

For a simplicial complex with a high degree of symmetry, we define a fold map which is similar to the fold map constructed for Theorem ~\ref{symDecomp}. It is shown that for polyhedral products of the form $(X,*)^K$, this induces a fold map on the level of polyhedral products. This gives us a special case of Theorem ~\ref{symDecomp}, presented in the paper as Theorem ~\ref{polysymDecomp}, in the case of polyhedral products. 

    Polyhedral products of the form $(\mathbb{C}P^{\infty},*)^K$ are known as \textit{Davis-Januszkiewicz} spaces which are denoted by $\djk$. These are closely related to polyhedral products of the form $(D^2,S^1)^K$ which are known as \textit{moment-angle complexes}, and these are denoted by $\zk$. There are various families of simplicial complexes for which it is known that the corresponding moment-angle complexes are homotopy equivalent to a wedge of simply connected spheres. These include shifted complexes \cite{GT2}, flag complexes with chordal 1-skeleton \cite{PT} and totally fillable complexes \cite{IK}. However, there are many examples of simplicial complexes for which the loops on the corresponding moment-angle complex are homotopy equivalent to a product of spheres and loops on simply connected spheres, but is not homotopy equivalent to a wedge of spheres before looping. In \cite{PT}, it was shown that the loops on a moment-angle complex associated to flag complexes (without the assumption on the 1-skeleton) are homotopy equivalent to a product of spheres and loops on simply connected spheres. The decomposition obtained here was a coarse description with no way to enumerate the spheres given in the product. In Section ~\ref{sec:bookgraphex}, we apply Theorem ~\ref{symDecomp} to give an explicit decomposition of the loops on a family of graphs known as book graphs.
    
The author would like to thank Stephen Theriault for his guidance and diligent proof reading during the preparation of this work.

\section{Preliminary homotopy theory}
\label{sec:homotopytheory}

In this section, we prove the results in homotopy theory that will be required later on. All the spaces in this section will be assumed to be path-connected $CW$-complexes. First, we have a result known as Mather's Cube Lemma \cite{M}.

\begin{theorem}
\label{cubelemma}
Suppose there is a homotopy commutative diagram of spaces and maps \[\begin{tikzcd}[row sep=scriptsize, column sep=scriptsize]
 E \arrow[dr]{} \arrow[rr] \arrow[dd] & & F \arrow[dr] \arrow[dd] &  \\
& G \arrow[rr, crossing over] \arrow[dd, crossing over] & & H \\
 A \arrow[dr] \arrow[rr] & & B \arrow[dr] &  \\
& C \arrow[rr] & & D \arrow[from=uu, crossing over]\\
\end{tikzcd}\] where the bottom face is a homotopy pushout and the four sides are homotopy pullbacks, then the top face is also a homotopy pushout.
\qedno
\end{theorem}

Let $\Omega B \xrightarrow{t} F \rightarrow E \xrightarrow{f} B$ be a homotopy fibration sequence. Let $b_0$ be the basepoint of $B$ and $PB = \{\omega:[0,1] \rightarrow B\: | \: \omega(1) = b_0\}$ be the path space of $B$. The homotopy fibre $F$ can be regarded as the space $P_f$ (c.f \cite{S}[p.59]) defined by \[P_f = \{(e,\omega) \in E \times PB | f(e) = \omega(0)\}.\] There is an action $\theta:\Omega B \times F \rightarrow F$ given by $\theta(\lambda,(e,\omega)) = (e,\lambda \cdot \omega)$ where $\lambda \cdot \omega$ is the path defined by \[\lambda \cdot \omega(t) = \begin{cases} \lambda(2t) & \text{if }0 \leq t \leq \frac{1}{2} \\ \omega(2t-1) & \text{if } \frac{1}{2} \leq t \leq 1.\end{cases}\] It follows from the definition that $\theta$ restricted to $F$ is homotopic to the identity and $\theta$ restricted to $\Omega B$ is homotopic to $t$. One property of this action is that the following diagram homotopy commutes \[\begin{tikzcd}
	{\Omega B \times F} & F \\
	F & E
	\arrow["{\pi_2}", from=1-1, to=2-1]
	\arrow["\theta", from=1-1, to=1-2]
	\arrow[from=1-2, to=2-2]
	\arrow[from=2-1, to=2-2]
\end{tikzcd}\] where $\pi_2$ is the projection into the second factor. With this action, the following result will be used to determine the homotopy type of various fibres.

\begin{lemma}
\label{Fibrethomotopyequiv}
Let $F \rightarrow E \rightarrow B$ and $F'\rightarrow E' \xrightarrow{s} B'$ be homotopy fibrations and suppose there exists a homotopy fibration diagram \[\begin{tikzcd}
	{\Omega F'} & G & F & {F'} \\
	{\Omega E'} & H & E & {E'} \\
	{\Omega B'} & J & B & {B'}
	\arrow[from=1-3, to=1-4]
	\arrow[from=1-1, to=1-2]
	\arrow["v", from=1-2, to=2-2]
	\arrow[from=1-3, to=2-3]
	\arrow[from=1-4, to=2-4]
	\arrow["s", from=2-4, to=3-4]
	\arrow[from=2-3, to=3-3]
	\arrow[from=2-3, to=2-4]
	\arrow[from=3-2, to=3-3]
	\arrow[from=2-2, to=2-3]
	\arrow[from=1-2, to=1-3]
	\arrow[from=1-1, to=2-1]
	\arrow[from=2-1, to=2-2]
	\arrow["t", from=3-1, to=3-2]
	\arrow["\Omega s", from=2-1, to=3-1]
	\arrow["u", from=2-2, to=3-2]
	\arrow[from=3-3, to=3-4]
\end{tikzcd}\] that defines the spaces $G$, $H$, $J$ and the maps $v:G \rightarrow H$ and $u: H \rightarrow J$. If $t \circ \Omega s$ has a right homotopy inverse $w:J \rightarrow \Omega E'$, then the composite \[\psi:J \times G \xrightarrow{w \times v} \Omega E' \times H \xrightarrow{\theta} H\] is a homotopy equivalence. 
\end{lemma}
\begin{proof}

First, consider the diagram \begin{equation}\label{splitting1stdiag}\begin{tikzcd}
	{J \times G} & {\Omega E' \times H} & H \\
	& {\Omega B' \times J} & J
	\arrow["\theta", from=1-2, to=1-3]
	\arrow["{w \times v}", from=1-1, to=1-2]
	\arrow["{r \times *}"', from=1-1, to=2-2]
	\arrow["{\Omega s \times u}", from=1-2, to=2-2]
	\arrow["{u}", from=1-3, to=2-3]
	\arrow["{\theta'}"', from=2-2, to=2-3]
\end{tikzcd}\end{equation} where $r = \Omega s \circ w$. The left triangle homotopy commutes since the composite $G \rightarrow H \rightarrow J$ is null homotopic and by definition of $r$. The right square homotopy commutes by the naturality of the homotopy action. The first row of (\ref{splitting1stdiag}) is the definition of $\psi$. The bottom direction of (\ref{splitting1stdiag}) is homotopic to the projection onto the first factor $\pi_1: J \times G \rightarrow J$ by definition of $r$ and since the restriction of $\theta'$ to $\Omega B'$ is homotopic to $t$. This gives us that $u \circ \psi \simeq \pi_1$. This homotopy results in the following homotopy fibration diagram that defines a map $\alpha: G \rightarrow G$
\begin{equation}\label{eq:5lemma1}\begin{tikzcd}
	G & G \\
	{J \times G} & H \\
	J & J
	\arrow["\alpha", from=1-1, to=1-2]
	\arrow["v", from=1-2, to=2-2]
	\arrow["u", from=2-2, to=3-2]
	\arrow["\psi", from=2-1, to=2-2]
	\arrow["i_2", from=1-1, to=2-1]
	\arrow["{\pi_1}", from=2-1, to=3-1]
	\arrow[Rightarrow, no head, from=3-1, to=3-2]
\end{tikzcd}\end{equation} where $i_2$ is the inclusion into the second factor. We want to show that $\alpha$ can be chosen to be the identity map.

The top square of (\ref{eq:5lemma1}) is a homotopy pullback. Since the restriction of $\theta$ to $H$ is homotopic to the identity, the composite $\psi \circ i_2$ is homotopic to $v$. Therefore, we obtain a pullback map $\beta: G \rightarrow G$ \[\begin{tikzcd}
	G \\
	& G & G \\
	& {J \times G} & H.
	\arrow["{i_2}", curve={height=12pt}, from=1-1, to=3-2]
	\arrow["\psi", from=3-2, to=3-3]
	\arrow["v", from=2-3, to=3-3]
	\arrow["\alpha", from=2-2, to=2-3]
	\arrow["{id_G}", curve={height=-12pt}, from=1-1, to=2-3]
	\arrow["i_2", from=2-2, to=3-2]
	\arrow["\beta", dashed, from=1-1, to=2-2]
\end{tikzcd}\] Now we wish to show that $\beta \simeq id_G$ which will imply that $\alpha \simeq id_G$. Consider the homotopy commutative diagram \[\begin{tikzcd}
	G \\
	& G \\
	& {J \times G} & G.
	\arrow["{id_G}", from=2-2, to=3-3]
	\arrow["\beta", from=1-1, to=2-2]
	\arrow["{i_2}", from=2-2, to=3-2]
	\arrow["{i_2}", curve={height=12pt}, from=1-1, to=3-2]
	\arrow["{\pi_2}", from=3-2, to=3-3]
\end{tikzcd}\] We can see from this that $id_G \circ \beta \simeq \pi_2 \circ i_2 = id_G$. Thus $\beta \simeq id_G$, implying that $\alpha \simeq id_G$. By the 5-lemma applied to (\ref{eq:5lemma1}), $\psi$ induces an isomorphism on homotopy groups. Therefore by Whitehead's theorem, $\psi$ is a homotopy equivalence.
\end{proof}
Let $i:A \hookrightarrow A \vee B$ be the inclusion into the first wedge summand and let $F$ be the homotopy fibre of $i$. In this case, the homotopy type of $F$ can be determined in terms of the homotopy fibre $F'$ of the pinch map $p:A \vee B \rightarrow A$. If $X$ and $Y$ are pointed spaces, the \textit{right half-smash} is defined by $X \rtimes Y = X \times Y /(*\times Y)$. It was shown in \cite{G} that there exists a homotopy equivalence $F' \simeq B \rtimes \Omega A$. The following proposition is a slight generalisation of \cite[Lemma 5.2]{HT1}.

\begin{proposition}
\label{inclusionfibre}
Let $A$ and $B$ be pointed spaces and let $i:A \hookrightarrow A \vee B$ be the inclusion. Denote by $F$ the homotopy fibre of $i$. Then there exists a homotopy equivalence $F \simeq \Omega(B \rtimes \Omega A)$.
\end{proposition}

\begin{proof}
Consider the following diagram \begin{equation}\label{incpinch}\begin{tikzcd}
	{*} & {B \rtimes \Omega A} \\
	A & {A \vee B} \\
	A & A.
	\arrow["p", from=2-2, to=3-2]
	\arrow["i",hook, from=2-1, to=2-2]
	\arrow[Rightarrow, no head, from=2-1, to=3-1]
	\arrow[Rightarrow, no head, from=3-1, to=3-2]
	\arrow[from=1-2, to=2-2]
	\arrow[from=1-1, to=1-2]
	\arrow[from=1-1, to=2-1]
\end{tikzcd}\end{equation} The bottom square commutes since the pinch map is a left inverse of $i$. Therefore, we can take homotopy fibres to obtain (\ref{incpinch}). In particular, the top square is a homotopy pullback. Therefore, $F$ is homotopy equivalent to the homotopy fibre of $* \rightarrow B \rtimes \Omega A$. Hence, there is a homotopy equivalence $F \simeq \Omega(B \rtimes \Omega A)$.
\end{proof}

Next, we will require the homotopy type of a specific homotopy pushout to decompose the homotopy fibre of a map in Section ~\ref{sec:application}. If $X$ and $Y$ are pointed spaces with basepoints $x_0$ and $y_0$ respectively, then the \textit{reduced join} is defined by $X * Y = (X \times I \times Y)/\sim$, where $(x,0,y) \sim (x,0,y')$, $(x,1,y) \sim (x',1,y)$ and $(x_0,t,y_0) \sim (x_0,0,y_0)$ for all $x,x' \in X$, $y,y' \in Y$ and $t \in I$. By definition, $A*B = A \times CB \cup_{A \times B} CA \times B$. The following result is from \cite{GT1}.

\begin{lemma}
\label{coordinatehtpyequiv}
Let $A$, $B$ and $C$ be spaces. Define $Q$ as the homotopy pushout \[\begin{tikzcd}
	{A \times B} & {C \times B} \\
	A & Q.
	\arrow["{\pi_1}", from=1-1, to=2-1]
	\arrow["{* \times id_B}", from=1-1, to=1-2]
	\arrow[from=1-2, to=2-2]
	\arrow[from=2-1, to=2-2]
\end{tikzcd}\] Then $Q \simeq (A*B) \vee (C \rtimes  B)$. 
\qedno
\end{lemma}

In Section ~\ref{sec:application} we will require a naturality property of the homotopy equivalence in Lemma ~\ref{coordinatehtpyequiv}. Suppose there are two homotopy pushouts $Q$ and $Q'$ of the form \begin{equation}\label{eq:pushoutsfornatural}\begin{tikzcd}
	{A \times B} & {C \times B} & {A' \times B'} & {C' \times B'} \\
	A & Q & {A'} & {Q'.}
	\arrow["{f \times id_B}", from=1-1, to=1-2]
	\arrow["{\pi_1}", from=1-1, to=2-1]
	\arrow[from=2-1, to=2-2]
	\arrow[from=1-2, to=2-2]
	\arrow["{\pi_1}", from=1-3, to=2-3]
	\arrow["{f' \times id_{B'}}", from=1-3, to=1-4]
	\arrow[from=2-3, to=2-4]
	\arrow[from=1-4, to=2-4]
\end{tikzcd}\end{equation} In addition, suppose $f$ and $f'$ are null homotopic, realised by homotopies $H:A \times I \rightarrow C$ and $H':A' \times I \rightarrow C'$. Suppose $H$ restricted to $A \times \{1\}$ and $H'$ restricted to $A' \times \{1\}$ are the respective constant maps. Let $a:A \xrightarrow{a} A'$, $b:B \xrightarrow{b} B'$ and $c: C \xrightarrow{c} C'$ be continuous maps, such that the following diagram strictly commutes \begin{equation}\label{eq:diagramsfornatural}\begin{tikzcd}
	{A \times I} & C \\
	{A' \times I} & {C'.}
	\arrow["H", from=1-1, to=1-2]
	\arrow["{H'}", from=2-1, to=2-2]
	\arrow["{a \times id_I}", from=1-1, to=2-1]
	\arrow["c", from=1-2, to=2-2]
\end{tikzcd}\end{equation} Since $H$ and $H'$ restricted to $A \times \{1\}$ and $A' \times \{1\}$ respectively are constant maps, there are quotient maps $\overline{f}$, $\overline{f}'$ such that (\ref{eq:diagramsfornatural}) can be written as a commutative diagram \begin{equation}\label{quotientdiagfornatural}\begin{tikzcd}
	CA & C \\
	{CA'} & {C'.}
	\arrow["{\overline{f}}", from=1-1, to=1-2]
	\arrow["{\overline{f}'}", from=2-1, to=2-2]
	\arrow["Ca", from=1-1, to=2-1]
	\arrow["c", from=1-2, to=2-2]
\end{tikzcd}\end{equation} The next lemma will show that the homotopy equivalence in Lemma ~\ref{coordinatehtpyequiv} is natural in this case. Let $f: A \rightarrow Y$ and $g: B \rightarrow Y$ be maps. The map $A \vee B \rightarrow Y$ determined by $f$ and $g$ is denoted by \[f \perp g: A \vee B \rightarrow Y.\]

\begin{lemma}
\label{naturalhtpyequiv}
Suppose we have data as in (\ref{eq:pushoutsfornatural}) and (\ref{eq:diagramsfornatural}). Then the homotopy equivalence in Lemma ~\ref{coordinatehtpyequiv} is natural with respect to $a$, $b$ and $c$.
\end{lemma}

\begin{proof}
Consider the homotopy pushout \[\begin{tikzcd}
	{A \times B} & {C \times B} \\
	A & Q.
	\arrow[from=1-2, to=2-2]
	\arrow[from=2-1, to=2-2]
	\arrow["{\pi_1}", from=1-1, to=2-1]
	\arrow["{f \times id_B}", from=1-1, to=1-2]
\end{tikzcd}\] Up to homotopy equivalences, replace the projection $A \times B \xrightarrow{\pi_1} A$ with the inclusion $A \times B \hookrightarrow A \times CB$ and define $\overline{Q}$ as the strict pushout \begin{equation}\label{strictnaturalpushout}\begin{tikzcd}
	{A \times B} & {C \times B} \\
	{A \times CB} & {\overline{Q}.}
	\arrow["{f \times id_B}", from=1-1, to=1-2]
	\arrow[hook , from=1-1, to=2-1]
	\arrow[from=2-1, to=2-2]
	\arrow["\beta", from=1-2, to=2-2]
\end{tikzcd}\end{equation} Note that $Q \simeq \overline{Q}$.

The map $f$ factors as the composite $A \hookrightarrow CA \xrightarrow{\overline{f}} C$, and so the pushout (\ref{strictnaturalpushout}) is equivalent to the iterated pushout \begin{equation}\label{iteratedpushoutnatural}\begin{tikzcd}
	{A \times B} & {CA \times B} & {C \times B} \\
	{A\times CB} & {A*B} & \overline{Q}
	\arrow[hook, from=1-1, to=2-1]
	\arrow[hook, from=1-2, to=2-2]
	\arrow[hook, from=2-1, to=2-2]
	\arrow["{\alpha}",from=2-2, to=2-3]
	\arrow["{\overline{f} \times id_B}", from=1-2, to=1-3]
	\arrow["{\beta}",from=1-3, to=2-3]
	\arrow[hook, from=1-1, to=1-2]
\end{tikzcd}\end{equation} which defines the map $\alpha$. Further, the inclusion $A \times B \hookrightarrow A \times CB$ factors as the composite $A \times B \hookrightarrow A \times B \cup_B CB \hookrightarrow A \times CB$. This implies that (\ref{iteratedpushoutnatural}) gives rise to the commutative diagram of pushouts \begin{equation}\label{iteratedpushforequiv}\begin{tikzcd}
	{A \times B} & {CA \times B} & {C \times B} \\
	{A \times B \cup_B CB} & M & N \\
	{A \times CB} & {A*B} & {\overline{Q}}
	\arrow["{\overline{f} \times id_B}", from=1-2, to=1-3]
	\arrow["\lambda", from=2-2, to=2-3]
	\arrow[from=1-3, to=2-3]
	\arrow[hook, from=1-2, to=2-2]
	\arrow[hook, from=1-1, to=1-2]
	\arrow[hook, from=1-1, to=2-1]
	\arrow[from=2-1, to=2-2]
	\arrow[from=2-1, to=3-1]
	\arrow[from=2-2, to=3-2]
	\arrow[hook, from=3-1, to=3-2]
	\arrow["\psi", from=2-3, to=3-3]
	\arrow["\alpha", from=3-2, to=3-3]
\end{tikzcd}\end{equation} defining spaces $M$ and $N$ and maps $\lambda$ and $\psi$. In particular, from the top squares in (\ref{iteratedpushforequiv}), $M = CA \times B \cup_{B} CB$ and $N = (C \times B) \cup_B CB$. Observe that in $N$, by contracting $CB$ to the cone point, we have a homotopy equivalence $\epsilon: N \xrightarrow{\simeq} C \rtimes B$. Now for $M$, by contracting $CB$ to the cone point, there is a homotopy equivalence $CA \times B \cup_{B} CB \simeq CA \rtimes B$. However, since $CA$ is contractible, it follows that $M$ is also contractible. Let $\overline{\gamma}:C \rtimes B \rightarrow \overline{Q}$ be the composite \[\overline{\gamma}:C \rtimes B \xrightarrow{\epsilon^{-1}} N \xrightarrow{\psi} \overline{Q}.\] Since $M \simeq *$ and $N \simeq C \rtimes B$, the bottom right square in (\ref{iteratedpushforequiv}) which is a pushout, implies that \[(A*B) \vee (C \rtimes B) \xrightarrow{\alpha \perp \overline{\gamma}} \overline{Q}\] is a homotopy equivalence.  

This recovers the result in Lemma ~\ref{coordinatehtpyequiv} but in a form that is better for proving naturality. All the maps in the two left squares of (\ref{iteratedpushforequiv}) are natural. By hypothesis, (\ref{quotientdiagfornatural}) commutes so $\overline{f}$ is natural with respect to such diagrams. The naturality of $\lambda$, $\psi$ and $\alpha$ then follow as they are induced pushout maps. The homotopy equivalence $\epsilon: N \rightarrow C \rtimes B$ is natural for maps $B \xrightarrow{b} B'$, $C \xrightarrow{c} C'$, giving a map $m:M \rightarrow M'$ and a homotopy commutative diagram \[\begin{tikzcd}
	M & {C \rtimes B} \\
	{M'} & {C' \rtimes B'.}
	\arrow["\epsilon", from=1-1, to=1-2]
	\arrow["{c \rtimes b}", from=1-2, to=2-2]
	\arrow["m", from=1-1, to=2-1]
	\arrow["\epsilon'", from=2-1, to=2-2]
\end{tikzcd}\]  This implies $\overline{\gamma}$ is natural as it is a composite of natural maps. Hence, the homotopy equivalence $(A*B) \vee (C \rtimes B) \xrightarrow{\alpha \perp \overline{\gamma}} \overline{Q}$ is natural.
\end{proof}

\section{The generalised fold map}
\label{sec:symsimpdecomp}

In this section, we define a generalised fold map and determine the homotopy type of its fibre. For $n \geq 2$, let $X_1,\cdots,X_n$ be homeomorphic, path-connected CW-complexes and for $1 \leq i < j \leq n$, let $\phi_{i,j}:X_i \rightarrow X_j$ be a homeomorphism. Let $A_1$, $A_2$ be homeomorphic subcomplexes of $X_1$. Also, assume there is an automorphism $\psi$ of $X_1$ such that $\psi(A_1) = A_2$ and $\psi(A_2) = A_1$. Define the space $P_2$ via the pushout \[\begin{tikzcd}
	{A_2} & {X_2} \\
	{X_1} & {P_2}
	\arrow["{f_2}", from=1-1, to=1-2]
	\arrow[from=1-2, to=2-2]
	\arrow[from=1-1, to=2-1]
	\arrow[from=2-1, to=2-2]
\end{tikzcd}\] where $f_2$ is the composite \[A_2 \xrightarrow{\psi|_{A_2}} A_1 \xrightarrow{\phi_{1,2}|_{A_1}} \phi_{1,2}(A_1) \hookrightarrow X_2.\] This corresponds to gluing $X_{1}$ to $X_2$ by gluing the copy of $A_{2}$ in $X_1$ to the copy of $A_{1}$ in $X_2$. Inductively define $P_n$ as the pushout \begin{equation}\label{definitionofPn}\begin{tikzcd}
	{\phi_{1,n-1}(A_{2})} & {X_n} \\
	{P_{n-1}} & {P_n}
	\arrow[from=2-1, to=2-2]
	\arrow[from=1-2, to=2-2]
	\arrow["{f_n}", from=1-1, to=1-2]
	\arrow[hook, from=1-1, to=2-1]
\end{tikzcd}\end{equation} where $f_n$ is the composite \[\phi_{1,n-1}(A_{2}) \xrightarrow{\phi_{1,n-1}^{-1}|_{\phi_{1,n-1}(A_2)}} A_2 \xrightarrow{\psi|_{A_2}} A_1 \xrightarrow{\phi_{1,n}|_{A_1}} \phi_{1,n}(A_1) \hookrightarrow X_n.\] This corresponds to gluing the copy of $X_{n-1}$ contained in $P_{n-1}$ to $X_n$ by gluing the copy of $A_{2}$ in $X_{n-1}$ to the copy of $A_{1}$ in $X_n$. We exploit the symmetry of the space $P_n$ to define a map $g_n$ called the \textit{fold map}. 

\begin{proposition}
\label{pushoutdefined}
For $n \geq 2$, there exists a pushout map $g_n$ making the following diagram commute \begin{equation}\label{eq:Fndefinition}\begin{tikzcd}
	{\phi_{1,n-1}(A_2)} & {X_n} \\
	{P_{n-1}} & {P_n} \\
	&& {X_1}
	\arrow[from=1-1, to=2-1]
	\arrow["f_n", from=1-1, to=1-2]
	\arrow[from=1-2, to=2-2]
	\arrow[from=2-1, to=2-2]
	\arrow["{g_{n-1}}", curve={height=12pt}, from=2-1, to=3-3]
	\arrow["{g_n}", from=2-2, to=3-3]
	\arrow["\psi_n \circ \phi_{1,n}^{-1}", curve={height=-12pt}, from=1-2, to=3-3]
\end{tikzcd}\end{equation} where $\psi_n$ is the identity on $X_1$ when $n$ is odd and $\psi_n = \psi^{-1}$ when $n$ is even. Moreover, $g_n$ maps $\phi_{1,n-1}(A_2)$ to $A_2$ when $n$ is even and to $A_1$ when $n$ is odd.
\end{proposition}
\begin{proof}
First consider the case $n=2$. By definition, $P_2$ is the pushout \[\begin{tikzcd}
	{A_2} & {X_2} \\
	{X_1} & {P_2.}
	\arrow["{f_2}", from=1-1, to=1-2]
	\arrow[from=1-2, to=2-2]
	\arrow[from=1-1, to=2-1]
	\arrow[from=2-1, to=2-2]
\end{tikzcd}\] Observe that by the definition of $\phi_{1,2}$ and $\psi$, the following diagram commutes \[\begin{tikzcd}
	{A_2} & {X_2} & {X_1} & {X_1} \\
	&&& {X_1}.
	\arrow["{f_2}", from=1-1, to=1-2]
	\arrow["{\phi_{1,2}^{-1}}", from=1-2, to=1-3]
	\arrow["\psi^{-1}", from=1-3, to=1-4]
	\arrow[Rightarrow, no head, from=1-4, to=2-4]
	\arrow[hook, from=1-1, to=2-4]
\end{tikzcd}\] Therefore, by the universal property of a pushout, we obtain a pushout map $g_2:P_2 \rightarrow X_1$ resulting in a commutative diagram \[\begin{tikzcd}
	{A_2} & {X_2} \\
	{X_1} & {P_2} \\
	&& {X_1}.
	\arrow[hook, from=1-1, to=2-1]
	\arrow["{f_2}", from=1-1, to=1-2]
	\arrow[from=1-2, to=2-2]
	\arrow[from=2-1, to=2-2]
	\arrow["{g_2}", from=2-2, to=3-3]
	\arrow["{\psi^{-1} \circ \phi_{1,2}^{-1}}", curve={height=-12pt}, from=1-2, to=3-3]
	\arrow["{id_{X_1}}", curve={height=12pt}, from=2-1, to=3-3]
\end{tikzcd}\]

Now suppose the result is true for $n-1$. First consider the case $n$ is odd and so $n-1$ is even. Consider the map $X_{n-1} \xrightarrow{\psi^{-1} \circ \phi_{1,n-1}^{-1}} X_1$ restricted to the subspace $\phi_{1,n-1}(A_2)$. The composite $\psi^{-1} \circ \phi_{1,n-1}^{-1}$ maps $\phi_{1,n-1}(A_2)$ to $A_1$ by the definition of $\psi$. Therefore, by the commutativity of \[\begin{tikzcd}
	{X_{n-1}} \\
	{P_{n-1}} & {X_1}
	\arrow[hook, from=1-1, to=2-1]
	\arrow["{g_{n-1}}", from=2-1, to=2-2]
	\arrow["{\psi^{-1}\circ \phi_{1,n-1}^{-1}}", from=1-1, to=2-2]
\end{tikzcd}\] in the inductive hypothesis, $\phi_{1,n-1}(A_2)$ is mapped to $A_1$ by $g_{n-1}$. Now, $\phi_{1,n-1}(A_2)$ is mapped to $\phi_{1,n}(A_1)$ in $X_n$ by $f_n$, as $f_n$ is the composite \[\phi_{1,n-1}(A_{2}) \xrightarrow{\phi_{1,n-1}^{-1}|_{\phi_{1,n-1}(A_2)}} A_2 \xrightarrow{\psi|_{A_2}} A_1 \xrightarrow{\phi_{1,n}|_{A_1}} \phi_{1,n}(A_1) \hookrightarrow X_n.\] Therefore, the composite $\phi^{-1}_{1,n} \circ f_n$ maps $\phi_{1,n-1}(A_1)$ to $A_1$ and so (\ref{eq:Fndefinition}) commutes in this case.

Now suppose $n$ is even and so $n-1$ is odd. Consider the map $X_{n-1} \xrightarrow{\phi_{1,n-1}^{-1}} X_1$ restricted to the subspace $\phi_{1,n-1}(A_2)$. The map $\phi_{1,n-1}^{-1}$ sends $\phi_{1,n-1}(A_2)$ to $A_2$. By the commutativity of \[\begin{tikzcd}
	{X_{n-1}} \\
	{P_{n-1}} & {X_1}
	\arrow[hook, from=1-1, to=2-1]
	\arrow["{g_{n-1}}", from=2-1, to=2-2]
	\arrow["{\phi_{1,n-1}^{-1}}", from=1-1, to=2-2]
\end{tikzcd}\] in the inductive hypothesis, $\phi_{1,n-1}(A_2)$ is mapped to $A_2$ by $g_{n-1}$. However, $\phi_{1,n-1}(A_2)$ is mapped to $\phi_{1,n}(A_1)$ in $X_n$ by $f_n$, and therefore the composite $\psi^{-1} \circ \phi^{-1}_{1,n} \circ f_n$ maps $\phi_{1,n-1}(A_2)$ to $A_2$ so (\ref{eq:Fndefinition}) commutes in this case.
\end{proof}

Intuitively, $g_n$ takes each copy of $X_i$ contained in $P_n$ and folds it down onto the copy of $X_1$. Now consider the homotopy fibration \[F_{P_n} \rightarrow P_n \xrightarrow{g_n} X_1\] that defines the space $F_{P_n}$. The composite \[X_1 \xrightarrow{\phi_{1,n} \circ \psi_n^{-1}} X_n \hookrightarrow P_n \xrightarrow{g_n} X_1\] is the identity map, since $\phi_{1,n} \circ \psi_n^{-1}$ is the inverse for the map $\psi_n \circ \phi_{1,n}^{-1}$ in (\ref{eq:Fndefinition}), and hence $g_n$ has a right homotopy inverse. This gives us the following result. 

\begin{lemma}
\label{Fsplitting}
The homotopy fibration $F_{P_n} \rightarrow P_n \xrightarrow{g_n} X_1$ splits after looping to give a homotopy equivalence \[\Omega P_n \simeq \Omega X_1 \times \Omega F_{P_n}.\]
\qedno
\end{lemma}

Recall from Proposition ~\ref{pushoutdefined} the definition of $g_n$ as the pushout map \begin{equation}\label{eq:usethecube}\begin{tikzcd}
	{\phi_{1,n-1}(A_2)} & {X_n} \\
	{P_{n-1}} & {P_n} \\
	&& {X_1.}
	\arrow["{g_n}", from=2-2, to=3-3]
	\arrow["{\psi_n^{-1} \circ \phi_{1,n}^{-1}}", curve={height=-12pt}, from=1-2, to=3-3]
	\arrow["{g_{n-1}}", curve={height=12pt}, from=2-1, to=3-3]
	\arrow[from=2-1, to=2-2]
	\arrow[from=1-2, to=2-2]
	\arrow["{f_n}", from=1-1, to=1-2]
	\arrow[hook, from=1-1, to=2-1]
\end{tikzcd}\end{equation} For $X$ one of $\phi_{1,n-1}(A_2)$, $X_n$, $P_{n-1}$, or $P_n$, let $F_X$ be the homotopy fibre of the composite $X \hookrightarrow P_n \xrightarrow{g_n} X_1$. Since all the homotopy fibres are given by composing into a common base, the following cube homotopy commutes \begin{equation}\label{cubeforpushoutfibres}\begin{tikzcd}[row sep=scriptsize, column sep=scriptsize]
 F_{\phi_{1,n-1}(A_2)} \arrow[dr]{} \arrow[rr] \arrow[dd] & & F_{X_n} \arrow[dr] \arrow[dd] &  \\
& F_{P_{n-1}} \arrow[rr, crossing over] \arrow[dd] & & F_{P_n} \\
 \phi_{1,n-1}(A_2) \arrow[rr] \arrow[dr] & & X_n \arrow[dr] &  \\
& P_{n-1} \arrow[rr] & & P_n \arrow[from=uu, crossing over]\\
\end{tikzcd}\end{equation} where the sides are homotopy pullbacks and the bottom face is a homotopy pushout. Therefore by Theorem \ref{cubelemma}, the top face is a homotopy pushout. We can identify the space $F_{P_n}$ in the top face of the cube more precisely.  
\begin{proposition}
\label{fibredecomp}
For $n \geq 2$, there is a homotopy equivalence \[F_{P_n} \simeq \bigvee\limits_{i=1}^{n-1} \Sigma F_{A_2}\] where $F_{A_2}$ is the homotopy fibre of the inclusion $A_2 \hookrightarrow X_1$. 
\end{proposition}

\begin{proof} This proof will be done in three steps. First, the $n=2$ case will be proved. Then, it will be shown that the induced map of fibres $F_{\phi_{1,n-1}(A_2)} \rightarrow F_{P_{n-1}}$ is null homotopic. Finally, this will allow us to inductively prove the proposition for $n \geq 3$.

\textit{Step 1: } By Proposition ~\ref{pushoutdefined}, we have a commutative diagram \[\begin{tikzcd}
	{A_1} & {X_2} \\
	{X_1} & {P_2} \\
	&& {X_1}.
	\arrow["{g_2}", from=2-2, to=3-3]
	\arrow["{f_2}", from=1-1, to=1-2]
	\arrow[hook, from=1-1, to=2-1]
	\arrow["{id_{X_1}}", curve={height=12pt}, from=2-1, to=3-3]
	\arrow[from=1-2, to=2-2]
	\arrow[from=2-1, to=2-2]
	\arrow["{\psi^{-1} \circ \phi_{1,2}^{-1}}", curve={height=-12pt}, from=1-2, to=3-3]
\end{tikzcd}\] By the top face of (\ref{cubeforpushoutfibres}) and the fact the homotopy fibre of a homeomorphism is contractible, we obtain a homotopy pushout of fibres \[\begin{tikzcd}
	{F_{A_2}} & {*} \\
	{*} & {F_{P_2}.}
	\arrow[from=1-1, to=1-2]
	\arrow[from=1-2, to=2-2]
	\arrow[from=1-1, to=2-1]
	\arrow[from=2-1, to=2-2]
\end{tikzcd}\] Therefore, $F_{P_2} \simeq \Sigma F_{A_2}$.

\textit{Step 2: } By construction, the map $F_{\phi_{1,n-1}(A_2)} \rightarrow F_{P_{n-1}}$ is the map of homotopy fibres induced by the bottom square in the following homotopy fibration diagram \begin{equation}\label{inducedmaptofactor}\begin{tikzcd}
	{F_{\phi_{1,n-1}(A_2)}} & {F_{P_{n-1}}} \\
	{\phi_{1,n-1}(A_2)} & {P_{n-1}} \\
	{X_1} & {X_1}
	\arrow["{g_{n-1}}", from=2-2, to=3-2]
	\arrow[Rightarrow, no head, from=3-1, to=3-2]
	\arrow[from=2-1, to=3-1]
	\arrow[from=1-1, to=1-2]
	\arrow[from=1-2, to=2-2]
	\arrow["\alpha",hook, from=2-1, to=2-2]
	\arrow[from=1-1, to=2-1]
\end{tikzcd}\end{equation} We consider the case when $n$ is odd and $n$ is even seperately. 

Suppose $n$ is even. By Proposition ~\ref{pushoutdefined}, the map $g_{n-1}$ maps $\phi_{1,n-1}(A_2)$ to $A_2$ in $X_1$. Consider the following diagram \begin{equation}\label{eq:nevenbigdiag}\begin{tikzcd}
	{F_{\phi_{1,n-1}(A_2)}} & {F_{P_{n-1}}} & {*} & {*} & {F_{P_{n-1}}} \\
	{\phi_{1,n-1}(A_2)} & {P_{n-1}} & {X_1} & {X_{n-1}} & {P_{n-1}} \\
	{X_1} & {X_1} & {X_1} & {X_1} & {X_1.}
	\arrow[from=2-1, to=3-1]
	\arrow[Rightarrow, no head, from=3-1, to=3-2]
	\arrow["\alpha",hook, from=2-1, to=2-2]
	\arrow["{g_{n-1}}", from=2-2, to=3-2]
	\arrow[from=1-1, to=1-2]
	\arrow[from=1-1, to=2-1]
	\arrow[from=1-2, to=2-2]
	\arrow["{g_{n-1}}", from=2-2, to=2-3]
	\arrow[Rightarrow, no head, from=3-2, to=3-3]
	\arrow[Rightarrow, no head, from=2-3, to=3-3]
	\arrow[from=1-2, to=1-3]
	\arrow[from=1-3, to=2-3]
	\arrow["{\phi_{1,n-1}}", from=2-3, to=2-4]
	\arrow[Rightarrow, no head, from=3-3, to=3-4]
	\arrow[from=1-3, to=1-4]
	\arrow[from=1-4, to=2-4]
	\arrow["{g_{n-1}}", from=2-5, to=3-5]
	\arrow[Rightarrow, no head, from=3-4, to=3-5]
	\arrow[from=1-4, to=1-5]
	\arrow[from=1-5, to=2-5]
	\arrow[hook, from=2-4, to=2-5]
	\arrow["{\phi_{1,n-1}^{-1}}", from=2-4, to=3-4]
\end{tikzcd}\end{equation}

Considering the bottom squares going from left to right: the first square commutes by definition, the second and third squares clearly commute, and the fourth square commutes by definition of $g_{n-1}$. Then take homotopy fibres vertically to obtain (\ref{eq:nevenbigdiag}). The middle row is equal to $\alpha$ and the bottom row is the identity on $X_1$. Each of the top squares is a homotopy pullback, and so the sequence of four homotopy pullbacks in the top part of (\ref{eq:nevenbigdiag}) is a homotopy pullback. Therefore, the top row shows that the induced map of fibres in (\ref{inducedmaptofactor}) factors through a contractible space, and so the map between fibres is null homotopic.

Now suppose $n$ is odd. By Proposition ~\ref{pushoutdefined}, the map $g_{n-1}$ maps $\phi_{1,n-1}(A_2)$ to $A_1$ in $X_1$. Consider the following diagram \begin{equation}\label{eq:noddbigdiag}\begin{tikzcd}
	{F_{\phi_{1,n-1}(A_2)}} & {F_{P_{n-1}}} & {*} & {*} & {*} & {F_{P_{n-1}}} \\
	{\phi_{1,n-1}(A_2)} & {P_{n-1}} & {X_1} & {X_1} & {X_{n-1}} & {P_{n-1}} \\
	{X_1} & {X_1} & {X_1} & {X_1} & {X_1} & {X_1.}
	\arrow["{g_{n-1}}", from=2-6, to=3-6]
	\arrow[from=1-6, to=2-6]
	\arrow[hook, from=2-5, to=2-6]
	\arrow["{\phi^{-1}_{1,n-1}}", from=2-5, to=3-5]
	\arrow["{\psi^{-1}}", from=3-5, to=3-6]
	\arrow[Rightarrow, no head, from=3-4, to=3-5]
	\arrow[from=1-5, to=1-6]
	\arrow[from=1-5, to=2-5]
	\arrow[from=1-4, to=1-5]
	\arrow[from=1-4, to=2-4]
	\arrow[Rightarrow, no head, from=2-4, to=3-4]
	\arrow["{\phi_{1,n-1}}", from=2-4, to=2-5]
	\arrow["\psi", from=2-3, to=2-4]
	\arrow["\psi", from=3-3, to=3-4]
	\arrow[from=1-3, to=1-4]
	\arrow[from=1-3, to=2-3]
	\arrow[Rightarrow, no head, from=2-3, to=3-3]
	\arrow["{g_{n-1}}", from=2-2, to=2-3]
	\arrow["{g_{n-1}}", from=2-2, to=3-2]
	\arrow[Rightarrow, no head, from=3-2, to=3-3]
	\arrow[from=1-2, to=2-2]
	\arrow[from=1-2, to=1-3]
	\arrow[from=1-1, to=1-2]
	\arrow[from=1-1, to=2-1]
	\arrow["\alpha",hook, from=2-1, to=2-2]
	\arrow[Rightarrow, no head, from=3-1, to=3-2]
	\arrow[from=2-1, to=3-1]
\end{tikzcd}\end{equation}

Considering the bottom squares going from left to right: the first square commutes by definition, the second, third and fourth squares clearly commute, and the fifth square commutes by definition of $g_{n-1}$. Then take homotopy fibres vertically to obtain (\ref{eq:noddbigdiag}). The middle row is equal $\alpha$ and the bottom row is the identity on $X_1$. Considering the top squares: the first and second squares are homotopy pullbacks and the outer perimeter of the third to fifth squares is a homotopy pullback, since the composite $\psi^{-1} \circ \psi \simeq id_{X_1}$. Therefore, the sequence of five homotopy pullbacks in the top part of (\ref{eq:noddbigdiag}) is a homotopy pullback. Hence, the top row shows that the induced map of fibres in (\ref{inducedmaptofactor}) factors through a contractible space, and so the map between fibres is null homotopic.

\textit{Step 3: } By (\ref{eq:usethecube}), the composite $X_n \hookrightarrow P_n \xrightarrow{g_n} X_1$ is equal to the composite $X_n \xrightarrow{\phi_{1,n}^{-1}} X_1 \xrightarrow{\psi_n^{-1}} X_1$, which is a homeomorphism. Therefore $F_{X_n}$ is contractible and so the homotopy pushout in the top face of (\ref{cubeforpushoutfibres}) \[\begin{tikzcd}
	{F_{\phi_{1,n-1}(A_2)}} & {F_{X_n}} \\
	{F_{P_{n-1}}} & {F_{P_{n}}}
	\arrow[from=1-1, to=1-2]
	\arrow[from=1-2, to=2-2]
	\arrow[from=1-1, to=2-1]
	\arrow[from=2-1, to=2-2]
\end{tikzcd}\] is equivalent, up to homotopy, to the homotopy pushout \[\begin{tikzcd}
	{F_{\phi_{1,n-1}(A_2)}} & {*} \\
	{F_{P_{n-1}}} & {F_{P_{n}}.}
	\arrow[from=1-1, to=2-1]
	\arrow[from=1-1, to=1-2]
	\arrow[from=1-2, to=2-2]
	\arrow[from=2-1, to=2-2]
\end{tikzcd}\]
This implies that $F_{P_{n}}$ is the homotopy cofibre of the map $F_{\phi_{1,n-1}(A_2)} \rightarrow F_{P_{n-1}}$. However, this map is null homotopic and so $F_{P_n} \simeq F_{P_{n-1}} \vee \Sigma F_{\phi_{1,n-1}(A_2)}$. Observe that the composite $\phi_{1,n-1}(A_2) \hookrightarrow P_{n-1} \xrightarrow{g_{n-1}} X_1$ is the same, up to homeomorphism, as the inclusion $A_2 \hookrightarrow X_1$. Therefore, there is a homotopy equivalence $F_{\phi_{1,i}(A_2)} \simeq F_{A_2}$ for all $2 \leq i \leq n-1$. By induction, we obtain \[F_{P_n} \simeq \bigvee\limits_{i=1}^{n-1} \Sigma F_{A_2}.\]
\end{proof}
We now have everything we need to prove Theorem ~\ref{symDecomp}.

\begin{proof}[Proof of Theorem ~\ref{symDecomp}:]
By Lemma ~\ref{Fsplitting}, there is a homotopy equivalence \begin{equation}\label{initialdecomp}\Omega P_n \simeq \Omega X_1 \times \Omega F_{P_n}\end{equation} where $F_{P_n}$ is the homotopy fibre of $P_n \xrightarrow{g_{n}} X_1$. By Proposition ~\ref{fibredecomp}, there is a homotopy equivalence \[F_{P_n} \simeq \bigvee\limits_{i=1}^{n-1} \Sigma F_{A_2}.\] Substituting this into (\ref{initialdecomp}), we obtain $\Omega P_n \simeq \Omega X_1 \times \Omega\left(\bigvee\limits_{i=1}^{n-1} \Sigma F_{A_2}\right)$ as desired.
\end{proof}

\begin{remark}
    As pointed out by an anonymous referee, there is an alternative, slick proof of Theorem ~\ref{symDecomp} using \cite[p.180]{F}. However, we leave the proof as presented as it is more hands-on and fits in well with the later applications. 
\end{remark}

\begin{remark}
The case where $X_i = X$ for $1 \leq i \leq n$ and $A_1 = A_2$ has many applications. In this case, we can take $\psi:X_1 \rightarrow X_1$ and $\phi_{i,j}:X_i \rightarrow X_j$ to be the identity map on $X$.
\end{remark}

\begin{example}
Let $X$ be a topological space. We apply Theorem ~\ref{symDecomp} to the wedge $\bigvee_{i=1}^n X$ to obtain a version of the Hilton-Milnor theorem. The wedge sum is formed by gluing the copies of $X$ together by the corresponding basepoints. Let $X_i = X$ for $1 \leq i \leq n$ and $A_1 = A_2 =*$. This means that the homotopy fibre $F_{A_2}$ in the statement of Theorem ~\ref{symDecomp} is the homotopy fibre of the inclusion $* \hookrightarrow X$. Hence $F_{A_2} \simeq \Omega X$. Applying Theorem ~\ref{symDecomp}, we obtain \[\Omega\left(\bigvee\limits_{i=1}^n X\right) \simeq \Omega X \times \Omega\left(\bigvee\limits_{i=1}^{n-1} \Sigma \Omega X\right).\] Denote by $(\Sigma \Omega)^k$ the operation $\overbrace{(\Sigma \Omega) \cdots (\Sigma \Omega)}^{k\text{ times}}$ with the convention that when $k=0$, the operator is the identity. By induction, we can further decompose $\Omega\left(\bigvee_{i=1}^{n-1} \Sigma \Omega X\right)$ to obtain \[\Omega\left(\bigvee\limits_{i=1}^n X\right) \simeq \prod\limits_{i=0}^{n-1} \Omega(\Sigma\Omega)^i X.\]
\end{example}

\begin{example}
We can also apply Theorem ~\ref{symDecomp} to the connected sum $X \conn X$, when $X$ is a closed $n$-dimensional manifold. Denote by $X\setminus D^n$ the manifold $X$ with an open disk removed. Let $S^{n-1}$ be the boundary circle of $D^n$ in $X$. We can think of $X \conn X$ as two copies of $X \setminus D^n$ glued together over the boundary circles $S^{n-1}$. Let $X_i = X\setminus D^n$ for $1 \leq i \leq 2$ and $A_1 = A_2 =S^{n-1}$. Let $G$ be the homotopy fibre of the inclusion $S^{n-1} \hookrightarrow X \setminus D^n$. Then applying Theorem ~\ref{symDecomp}, we obtain \[\Omega (X \conn X) \simeq \Omega (X\setminus D^n) \times \Omega\Sigma G.\] The manifold $X \conn X$ is a special case of the double of a manifold with boundary. Loop space decompositions of the double of a manifold with boundary were considered in \cite[Section 2]{HT2}, and this recovers a special case of \cite[Lemma 2.1]{HT2}.

\end{example}

\section{Properties of Polyhedral Products}
\label{sec:basicpropPP}
In this section, we introduce the basic properties of polyhedral products needed to apply Theorem ~\ref{symDecomp} to them. First, we will require the notion of two simplicial complexes being \textit{isomorphic}.
\begin{definition}
\label{simpiso}
Let $K_1$ and $K_2$ be simplicial complexes. A \textit{simplicial isomorphism} is a bijective map $f:K_1 \rightarrow K_2$ such that if $\sigma$ is a face of $K_1$, then $f(\sigma)$ is a face of $K_2$ for all $\sigma \in K_1$.
\end{definition} 

Two simplicial complexes are called \textit{isomorphic} if there exists a simplicial isomorphism between $K_1$ and $K_2$. This definition means that two simplicial complexes are isomorphic if there exists a relabelling of $K_2$ which makes it identical to $K_1$.

Let $\mathcal{K}$ be the category of simplicial complexes with morphisms which are simplicial maps and $CW_*$ be the category of connected, pointed $CW$-complexes and pointed continuous maps. The polyhedral product is a functor from $\mathcal{K}$ to $CW_*$ \cite{BBCG}. From this, we obtain the following results.

\begin{theorem}
\label{inducedinclusion}
Let $K$ and $L$ be simplicial complexes on the vertex sets $[m]$ and $[n]$ respectively. Suppose there is a simplicial inclusion $K\xrightarrow{f} L$. Then for any sequence $\uxa$ of pointed, path connected $CW$-complexes, there is an induced map $\uxa^K \rightarrow \uxa^L$.
\qedno
\end{theorem}

\begin{theorem}
\label{isoPP}
Let $K_1$ and $K_2$ be isomorphic simplicial complexes on the vertex set $[m]$. Let $(X,A)$ be a pointed $CW$-pair. Then there exists a homeomorphism $(X,A)^{K_1} \cong (X,A)^{K_2}$.
\qedno
\end{theorem}

There is a second type of map we can define between polyhedral products. This involves a special subcomplex known as a full subcomplex.

\begin{definition}
Let $K$ be a simplicial complex on the vertex set $[m]$. If $I \subseteq [m]$, then the \textit{full subcomplex} $K_I$ of $K$ is defined as the simplicial complex \[K_I = \bigcup\{\sigma \in K \: | \: \text{the vertex set of } \sigma \text{ is a subset of } I\}.\]
\end{definition}

In the special case where $K$ is a graph, $K_I$ is known as the \textit{induced subgraph} on $I$. Since $K_I$ is a subcomplex of the simplicial complex $K$, there is an inclusion $\uxa^{K_I} \rightarrow \uxa^K$ of polyhedral products. Working the other way, projecting from $[m]$ to $I$ does not induce a map of simplicial complexes $K \rightarrow K_I$. However, there is a projection on the level of polyhedral products. If $I = \{i_1,\cdots,i_k\}$ for $1 \leq i_1 < \cdots <i_k \leq m$, let $X^I = \prod_{j=1}^k X_{i_j}$. The following result from \cite{DS} shows that $\uxa^{K_I}$ retracts off $\uxa^K$.

\begin{proposition}
\label{fullretract}
Let $K$ be a simplicial complex on the vertex set $[m]$ and let $\uxa$ be any sequence of pointed, path-connected CW-pairs. Let $I \subseteq [m]$. Then the projection $\prod_{i=1}^m X_i \rightarrow X^I$ induces a map $\uxa^K \rightarrow \uxa^{K_I}$. Further, the composite $\uxa^{K_I} \hookrightarrow \uxa^K \rightarrow \uxa^{K_I}$ is the identity map.
\qedno
\end{proposition}

The next property is the relation between pushouts of simplicial complexes and pushouts of polyhedral products. Let $K$ be a simplicial complex on the vertex set $[m]$. Suppose there is a pushout of simplicial complexes \[\begin{tikzcd}
	L & {K_1} \\
	{K_2} & K.
	\arrow[from=1-1, to=1-2]
	\arrow[from=1-2, to=2-2]
	\arrow[from=1-1, to=2-1]
	\arrow[from=2-1, to=2-2]
\end{tikzcd}\] To compare the polyhedral products for $L$, $K_1$, $K_2$ and $K$, we should consider them all over the same vertex set $[m]$. We can do this by introducing ghost vertices. Considering $L$, $K_1$ and $K_2$ as simplcial complexes over $[m]$, a \textit{ghost vertex} is a one element subset of $[m]$ that is not in the given simplicial complex. Denote the simplicial complexes over the vertex set $[m]$ by $\bar{L}$, $\bar{K}_1$, $\bar{K}_2$. In this case, $K = \bar{K}_1 \cup_{\bar{L}} \bar{K}_2$. The following result is from \cite{GT2}.

\begin{proposition}
\label{inducepushout}
Let $K$ be a simplicial complex on the vertex set $[m]$. Suppose there is a pushout of simplicial complexes \[\begin{tikzcd}
	L & {K_1} \\
	{K_2} & K.
	\arrow[from=1-1, to=1-2]
	\arrow[from=1-2, to=2-2]
	\arrow[from=1-1, to=2-1]
	\arrow[from=2-1, to=2-2]
\end{tikzcd}\] Then there is a pushout of polyhedral products \[\begin{tikzcd}
	{(\underline{X},\underline{A})^{\bar{L}}} & {(\underline{X},\underline{A})^{\bar{K}_1}} \\
	{(\underline{X},\underline{A})^{\bar{K}_2}} & {(\underline{X},\underline{A})^K}.
	\arrow[from=1-1, to=1-2]
	\arrow[from=1-1, to=2-1]
	\arrow[from=2-1, to=2-2]
	\arrow[from=1-2, to=2-2]
\end{tikzcd}\]
\qedno
\end{proposition}

In \cite{DS}, the following theorem was proved which allows us to determine the homotopy type of $\Omega(\underline{X},\underline{*})^K$ from $\Omega(\underline{C\Omega X},\underline{\Omega X})^{K}$.

\begin{theorem}
\label{coneloop}
Let $K$ be a simplicial complex on the vertex set $[m]$ and let $\{(\underline{X},\underline{*})\}_{i=1}^m$ be a sequence of pointed pairs $(X_i,\ast)$ where each $X_i$ is path-connected. Then there is a homotopy fibration \[\clxx^K \rightarrow (\underline{X},\underline{*})^K \rightarrow \prod\limits_{i=1}^m X_i.\]

Further, this fibration splits after looping to give a homotopy equivalence \[\Omega (\underline{X},\underline{*})^K \simeq \left(\prod\limits_{i=1}^m \Omega X_i \right) \times \Omega \clxx^K.\]
\qedno
\end{theorem}

The following result will be required in Section ~\ref{sec:bookgraphex} to apply Lemma ~\ref{naturalhtpyequiv}. This is an adaptation of a result in \cite{GT2}, but a different proof will be provided as we will require an explicit homotopy for our application.

\begin{lemma}
\label{conenullhtpy}
Let $\{v_1,\cdots,v_n\}$ be a set of disjoint points. Then the inclusion \[\prod_{i=1}^{n-1} X_i \times CX_n \xhookrightarrow{k} (\underline{CX},\underline{X})^{\{v_1,\cdots,v_n\}}\] is null homotopic. Moreover, the homotopy can be chosen such that, up to reparameterisation, its restriction to $X_i$ for $1 \leq i \leq n-1$ is $X_i \times I \xrightarrow{H} CX_i$ where $H$ is the standard quotient map sending $X_i \vee I$ to a point.
\end{lemma}
\begin{proof}
By definition, $(\underline{CX},\underline{X})^{\{v_1,\cdots,v_n\}} = \bigcup_{i=1}^n X_1 \times \cdots \times CX_i \times \cdots \times X_n$. For $1 \leq i \leq n-1$, let $H_i:X_i \times I \rightarrow CX_i$ be the standard quotient map sending $X_i \vee I$ to a point and $H_n: CX_n \times I \rightarrow CX_n$ be the homotopy which contracts the cone to the cone point. Let the restriction of $H_i$ for $1 \leq i \leq n-1$ to $X_i \times \{1\}$ and the restriction of $H_n$ to $CX_n \times \{1\}$ be the respective constant maps. Define a homotopy $H:\prod_{i=1}^{n-1} X_i \times CX_n \times I \rightarrow \bigcup_{i=1}^n X_1 \times \cdots \times CX_i \times \cdots \times X_n$ by \[H((x_1,\cdots,x_{n-1},(x_n,t')),t) = \begin{cases} (x_1,x_2,\cdots,x_{n-1},H_n((x_n,t'),nt) & t \in [0,\frac{1}{n}] \\
(H_1(x_1,nt-1),x_2,x_3,\cdots,x_{n-1},*) & t \in [\frac{1}{n},\frac{2}{n}]\\
(*,H_2(x_2,nt-2),x_3,\cdots,x_{n-1},*) & t \in [\frac{2}{n},\frac{3}{n}]\\[-1.5ex]
\vdotswithin{ = } &\vdotswithin{ = } \\[-1ex]
(*,*,\cdots,*,H_i(x_i,nt-i),x_{i+1},\cdots,x_{n-1},*) & t \in [\frac{i}{n},\frac{i+1}{n}]\\[-1.5ex]
\vdotswithin{ = } &\vdotswithin{ = } \\[-1ex]
(*,*,\cdots,*,H_{n-1}(x_{n-1},nt-(n-1))) & t \in [\frac{n-1}{n},1].
\end{cases}\] This is well defined since at each time $t$, only one $x_i$ is not in the base of $CX_i$, at $t=\frac{1}{n}$, $H_n((x_n,t'),t) = *$ and at $t=\frac{i}{n}$, $H_{i-1}(x_{i-1},\frac{ni}{n}-(i-1)) = *$ and $H_{i}(x_{i},\frac{ni}{n}-i) = x_{i}$. At $t=0$, $H(x,0) = k(x)$ and $H(x,1) = *$. Therefore, $H$ is a homotopy between $k$ and the constant map.
\end{proof}

Now, we show that there is a fold map on the level of simplicial complexes that induces a fold map, as defined in Section ~\ref{sec:symsimpdecomp}, in the case of polyhedral products of the form $\ux^K$. For $n\geq 2$, let $K_1, K_2, \cdots, K_n$ be isomorphic simplicial complexes on a common vertex set $[m]$ and for $1 \leq i <j \leq n$, let $\phi_{i,j}:K_i \rightarrow K_j$ be a simplicial isomorphism. Let $L_1$, $L_2$ be isomorphic subcomplexes of $K_1$ and assume there is a simplicial automorphism $\psi:K_1 \rightarrow K_1$ such that $\psi(L_1) = L_2$ and $\psi(L_2) = L_1$. Define $M_2$ by the pushout \[\begin{tikzcd}
	{L_2} & {K_2} \\
	{K_1} & {M_2}
	\arrow["{f_2}", from=1-1, to=1-2]
	\arrow[from=1-2, to=2-2]
	\arrow[hook, from=1-1, to=2-1]
	\arrow[from=2-1, to=2-2]
\end{tikzcd}\] where $f_2$ is the composite \[L_2 \xrightarrow{\psi|_{L_2}} L_1 \xrightarrow{\phi_{1,2}|_{L_1}} \phi_{1,2}(L_1) \hookrightarrow K_2.\] This corresponds to gluing the copy of $L_2$ in $K_1$ to the copy of $L_1$ in $K_2$. Inductively, define $M_n$ by the pushout \[\begin{tikzcd}
	{\phi_{1,n-1}(L_2)} && {K_n} \\
	{M_{n-1}} && {M_n}
	\arrow["{f_n}", from=1-1, to=1-3]
	\arrow[from=1-3, to=2-3]
	\arrow[hook, from=1-1, to=2-1]
	\arrow[from=2-1, to=2-3]
\end{tikzcd}\] where $f_n$ is the composite \[\phi_{1,n-1}(L_{2}) \xrightarrow{\phi_{1,n-1}^{-1}|_{\phi_{1,n-1}(L_2)}} L_2 \xrightarrow{\psi|_{L_2}} L_1 \xrightarrow{\phi_{1,n}|_{L_1}} \phi_{1,n}(L_1) \hookrightarrow K_n.\] This corresponds to gluing the copy of $L_2$ in $K_{n-1}$ which is contained in $M_{n-1}$ to the copy of $L_1$ in $K_n$.

The map $f_n$ is the composite of simplicial isomorphisms followed by an inclusion. Hence, $f_n$ is injective and so is also a simplicial inclusion. Therefore by Proposition ~\ref{inducepushout}, the pushout defining $M_n$ induces a pushout of polyhedral products where since $A_i = *$, the extra terms in the pushout associated with the ghost vertices disappear \[\begin{tikzcd}
	{(X,*)^{\phi_{1,n-1}(L_2)}} & {(X,*)^{K_n}} \\
	{(X,*)^{M_{n-1}}} & {(X,*)^{M_n}}.
	\arrow[from=2-1, to=2-2]
	\arrow[from=1-1, to=2-1]
	\arrow[from=1-2, to=2-2]
	\arrow[from=1-1, to=1-2]
\end{tikzcd}\] This pushout is of the form as in (\ref{definitionofPn}) required by Theorem ~\ref{symDecomp}. Therefore, we obtain the following.

\begin{theorem}
\label{polysymDecomp}
Let $M_n$ be a simplicial complex constructed as above where $n \geq 2$. Let $X$ be a path-connected pointed $CW$-complex. Let $G$ be the homotopy fibre of the inclusion $\ux^{L_2} \hookrightarrow \ux^{K_1}$. Then there is a homotopy equivalence \[\Omega\ux^{M_n} \simeq \Omega\ux^{K_1} \times \Omega F\] where $F \simeq \bigvee\limits_{i=1}^{n-1} \Sigma G$.
\qedno
\end{theorem}

\begin{remark}
As in Theorem ~\ref{symDecomp}, in the case where $K_i = K$ for $1 \leq i \leq n$ and $L_1 = L_2$, we can take $\psi:K_1 \rightarrow K_1$ and $\phi_{i,j}:K_i \rightarrow K_j$ to be the identity map on $K$.
\end{remark}

\section{Applications to polyhedral products associated to graphs}
\label{sec:application}

\subsection{Background on Graph Theory}
\label{sec:backgroundGT}

A graph $G$ is a set of vertices $V(G)$ together with a set of edges $E(G)$ that connect them. In general, this is not a valid simplicial complex as it may have an edge which connects a vertex to itself, or may have multiple edges connecting two vertices $v_0$ and $v_1$. Graphs which do not have either of these are known as \textit{simple} graphs. A graph $G$ is \textit{finite} if its vertex set $V(G)$ and edge set $E(G)$ are finite. A subgraph of $G$ is a graph $H$ such that $V(H) \subseteq V(G)$ and $E(H) \subseteq V(G)$.

There are various notions of adjacency in a graph. Denote the edge connecting two vertices $v$ and $v'$ by $v \cdot v'$. Two vertices $v, v' \in V(G)$ are \textit{adjacent} if there is an edge $e \in E(G)$ such that $e = v \cdot v'$. Two edges $e,e' \in E(G)$ are \textit{adjacent} if they share an end vertex. Finally, a vertex $v \in V(G)$ and $e \in E(G)$ are \textit{adjacent} if $v$ is an end vertex of $e$. 

In most cases, we will consider graphs which are \textit{connected}. This means for any pair of vertices $v$, $w \in V(G)$, there exists a sequence of vertices $v= v_0,v_1, \cdots, v_m = w$ such that $v_i$ is adjacent to $v_{i+1}$ for all $1 \leq i \leq n-1$.

We will consider many special types of graphs. 

\begin{definition}
The \textit{degree} of a vertex $v$ is the number of edges adjacent to $v$.
\end{definition}

\begin{definition}
A \textit{cycle} in a graph $G$ is a sequence of vertices $v_0,v_1,\cdots,v_m$ such that $v_0 = v_m$ and $v_i$ is adjacent to $v_{i+1}$ for all $1 \leq i \leq n-1$.
\end{definition}

\begin{definition}
A graph $G$ is a \textit{tree} if it contains no cycles.
\end{definition}

\begin{definition}
A \textit{path graph} $P_n$ is a tree on $n+1$ vertices with 2 vertices of degree 1 and $n-1$ vertices of degree 2.
\end{definition}

For example, $P_4$ is the following graph

\[\scalebox{1.5}{\begin{tikzpicture}
     \draw (0.0,-0.1) -- (1.2,-0.1); 
      
     \draw [fill] (0.0,-0.1) circle [radius=0.03];
		 \draw [fill] (0.3,-0.1) circle [radius=0.03];
		 \draw [fill] (0.6,-0.1) circle [radius=0.03];
		 \draw [fill] (0.9,-0.1) circle [radius=0.03];
		 \draw [fill] (1.2,-0.1) circle [radius=0.03];
  \end{tikzpicture}}\] 

\begin{definition}
A \textit{cycle graph} $C_L$ of length $L \geq 3$ is a graph with $L$ vertices and $L$ edges which contains a single cycle of length $L$.
\end{definition}

For example, the graph $C_4$ is the following graph

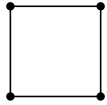
\[\scalebox{1.5}{\begin{tikzpicture}
     \draw (0.0,-0.1) -- (0.8,-0.1) -- (0.8,-0.9) -- (0.0,-0.9) -- (0.0,-0.1); 
      
     \draw [fill] (0.0,-0.1) circle [radius=0.03];
		 \draw [fill] (0.8,-0.1) circle [radius=0.03];
		 \draw [fill] (0.8,-0.9) circle [radius=0.03];
		 \draw [fill] (0.0,-0.9) circle [radius=0.03];
  \end{tikzpicture}}\] Many of the examples of graphs we will consider will be built from these. 

\subsection{Polyhedral products associated to generalised book graphs}
\label{sec:bookgraphex}

The generalised book graph is denoted $B(n,l,p)$ where $1 \leq n \leq l-2$, $l \geq 3$ and $p \geq 2$. The graph $B(n,l,p)$ is $p$ cycles $C_l$ of length $l$ glued together over a common path $P_n$ of length $n$. Formally, define $B(n,l,2)$ as the pushout \[\begin{tikzcd}
	{P_n} & {C_l} \\
	{C_l} & {B(n,l,2).}
	\arrow[hook, from=1-1, to=1-2]
	\arrow[hook, from=1-1, to=2-1]
	\arrow[from=1-2, to=2-2]
	\arrow[from=2-1, to=2-2]
\end{tikzcd}\] Then for $p >2$, define $B(n,l,p)$ iteratively as the pushout \begin{equation}\label{bookgraphdefinition}\begin{tikzcd}
	{P_n} & {C_l} \\
	{B(n,l,p-1)} & {B(n,l,p).}
	\arrow[hook, from=1-1, to=1-2]
	\arrow[hook, from=1-1, to=2-1]
	\arrow[from=1-2, to=2-2]
	\arrow[from=2-1, to=2-2]
\end{tikzcd}\end{equation} For example, $B(1,3,2)$ is the graph 
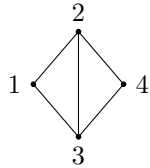
\[\begin{tikzpicture} 
	   \draw (-0.1,-0.1) -- (0.5,0.6) -- (0.5,-0.8) -- (-0.1,-0.1);
     \draw (0.5,0.6) -- (1.1,-0.1) -- (0.5,-0.8);
		 \draw [fill] (-0.1,-0.1) circle [radius=0.03];
		 \draw [fill] (0.5,0.6) circle [radius=0.03];
		 \draw [fill] (0.5,-0.8) circle [radius=0.03];
		 \draw [fill] (1.1,-0.1) circle [radius=0.03];
	   \node at (-0.35,-0.1) {1};
		 \node at (0.5,0.85) {2};
		 \node at (0.5,-1.05) {3};
		 \node at (1.35,-0.1) {4};
  \end{tikzpicture}\] and the graph $B(1,3,3)$ is \[\scalebox{1.5}
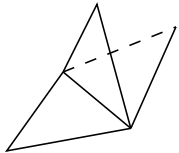
{\begin{tikzpicture} 
     \draw (-0.1,-0.1) -- (0.5,-0.6); 
		 \draw (-0.1,-0.1) -- (0.2,0.5) -- (0.5,-0.6); 
		 \draw (0.9,0.3) -- (0.5,-0.6); 
		 \draw [dashed] (-0.1,-0.1) -- (0.9,0.3);
		 \draw (-0.1,-0.1) -- (-0.6,-0.8) -- (0.5,-0.6); 
  \end{tikzpicture}}\] 

We can apply Theorem ~\ref{polysymDecomp} to $B(n,l,p)$ using (\ref{bookgraphdefinition}). Let $L_1 = L_2 = P_n$, $K_i = C_l$ for $1 \leq i \leq p$ and $\psi$ and $\phi_{1,j}$ be the identity map for $2 \leq j \leq p$. By Theorem ~\ref{polysymDecomp}, we obtain a homotopy equivalence \[\Omega \ux^{B(n,l,p)} \simeq \Omega \ux^{C_l} \times \Omega \left(\bigvee\limits_{i=1}^{p-1} \Sigma G\right)\] where $G$ is the homotopy fibre of the inclusion $\ux^{P_n} \hookrightarrow \ux^{C_l}$. If we specialise to generalised book graphs of the form $B(l,2l,p)$, we can determine the homotopy type of $\Omega \ux^{B(l,2l,p)}$ more precisely.

For the decomposition that follows, we draw the generalised book graph in its planar form. For example $B(2,4,4)$ is a graph of the form  \[\scalebox{1.5}
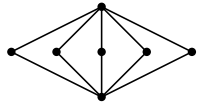
{\begin{tikzpicture}
     \draw (-0.1,-0.1) -- (-0.1,-0.9); 
		 \draw (-0.1,-0.1) -- (-0.5,-0.5) -- (-0.1,-0.9); 
		 \draw (-0.1,-0.1) -- (0.3,-0.5) -- (-0.1,-0.9); 
		 \draw (-0.1,-0.1) -- (0.7,-0.5) -- (-0.1,-0.9);
		 \draw (-0.1,-0.1) -- (-0.9,-0.5) -- (-0.1,-0.9);
      
     \draw [fill] (-0.1,-0.1) circle [radius=0.03];
		 \draw [fill] (-0.1,-0.5) circle [radius=0.03];
		 \draw [fill] (-0.1,-0.9) circle [radius=0.03];
		 \draw [fill] (-0.5,-0.5) circle [radius=0.03];
		 \draw [fill] (0.3,-0.5) circle [radius=0.03];
		 \draw [fill] (-0.9,-0.5) circle [radius=0.03];
		 \draw [fill] (0.7,-0.5) circle [radius=0.03];
  \end{tikzpicture}}\] In this form, the graph $B(l,2l,p)$ consists of $p+1$ paths of length $l$ glued together over their endpoints. Write $B(l,2l,p)$ as the pushout \[\begin{tikzcd}
	{\{v_0,v_l\}} & {P_l} \\
	{B(l,2l,p-1)} & {B(l,2l,p)}
	\arrow[from=1-1, to=1-2]
	\arrow[from=1-1, to=2-1]
	\arrow[from=1-2, to=2-2]
	\arrow[from=2-1, to=2-2]
\end{tikzcd}\] where $v_0,v_l$ are the endpoints of $P_l$. Let $K_i = P_l$, $L_1 = L_2 = \{v_0,v_l\}$ and $\psi$ and $\phi_{1,j}$ be the identity map for $2 \leq j \leq p+1$. Applying Theorem ~\ref{polysymDecomp}, we obtain an alternative formulation of a homotopy decomposition for $\Omega \ux^{B(l,2l,p)}$ as \[\Omega \ux^{B(l,2l,p)} \simeq \Omega\ux^{P_l} \times \Omega \left(\bigvee\limits_{i=1}^{p} \Sigma F\right)\] where $F$ is the homotopy fibre of the inclusion $\ux^{\{v_0,v_l\}} \hookrightarrow \ux^{P_l}$. In this case, we can identify $\Omega \ux^{P_l}$ and $F$. The following result from \cite{T} allows us to determine the homotopy type of $\Omega\ux^{P_l}$.

\begin{proposition}
\label{Pathisdisjointvertex}
    Let $k \geq 1$ and suppose that there is a sequence of simplicial complexes \[K_1 = \Delta^k \subseteq K_2 \subseteq \cdots \subseteq K_{\ell}\] such that, for $i > 1$, $K_i = K_{i-1} \cup_{\sigma_i} \Delta^k$ where $\sigma_i = \Delta^{k-1}$. Let $K = K_\ell$ and observe that $K$ is a simplicial complex on $k +\ell$ vertices. Then there is a homotopy equivalence \[(CX,X)^K \simeq (CX,X)^{V_\ell}\] where $V_\ell$ is $\ell$ disjoint points.
\end{proposition}

We can determine $\Omega\ux^{P_l}$ by using Theorem ~\ref{coneloop}, which gives us that \[\Omega\ux^{P_l} \simeq \left(\prod\limits_{i=0}^l \Omega X \right) \times \Omega (C\Omega X,\Omega X)^{P_{l}}.\] The path $P_l$ can be constructed interatively by gluing 1-simplices along a vertex. Therefore, Proposition ~\ref{Pathisdisjointvertex} implies that there is a homotopy equivalence $(C\Omega X,\Omega X)^{P_l} \simeq (C \Omega X,\Omega X)^{V_l}$. Now consider the homotopy fibration \[(C\Omega X,\Omega X)^{V_l} \rightarrow (X,*)^{V_l} \rightarrow \prod\limits_{i=1}^l \Omega X_i\] in Theorem ~\ref{coneloop} for the case $K = V_l$. By definition, $(X,*)^{V_l} = \bigvee_{i=1}^l X$, and so $(C\Omega X,\Omega X)^{V_l}$ is the homotopy fibre of the inclusion $\bigvee_{i=1}^l X \hookrightarrow \prod_{i=1}^l \Omega X$. Porter \cite{P} identified the homotopy type of this fibre as \[\bigvee\limits_{k=2}^l \bigvee\limits_{1 \leq i_1 < \cdots < i_k \leq l} (\Sigma \Omega X_{i_1} \wedge \cdots \wedge \Omega X_{i_k})^{\vee(k-1)}\] where in our case, each $X_{i_j} = X$. Let $\mathcal{W}$ be the set of topological spaces homotopy equivalent to a wedge of spaces where each summand is a suspension of smashes of $\Omega X_i$'s. Porter's identification of the homotopy fibre gives us the following result.

\begin{lemma}
\label{Pathexplicithtpyequiv}
    There is a homotopy equivalence \[(C\Omega X,\Omega X)^{P_l} \simeq \bigvee\limits_{k=2}^l \bigvee\limits_{1 \leq i_1 < \cdots < i_k \leq l} (\Sigma \Omega X_{i_1} \wedge \cdots \wedge \Omega X_{i_k})^{\vee(k-1)}.\] In particular, $(C\Omega X,\Omega X)^{P_l} \in \mathcal{W}$.
    \qedno
\end{lemma}

Next we determine $F$. First consider the path of length 2. Let the path be $v_0 \cdot v_1 \cdot v_2$. We wish to identify the homotopy fibre of the inclusion $\ux^{\{v_0,v_2\}} \hookrightarrow \ux^{P_2}$. In this case, the path $P_2$ decomposes as the join $\{v_0,v_2\} *\{v_1\}$ so $\ux^{P_2} \cong \ux^{\{v_0,v_2\}} \times \ux^{\{v_1\}}$. By definition of the polyhedral product, this is equivalent to identifying the homotopy fibre of the inclusion \[X \vee X \hookrightarrow (X \vee X) \times X\] into the first factor. Therefore, $F \simeq \Omega X$ in this case.

We now consider the case where the length of the path is greater than or equal to 3. Consider the following fibration diagram which defines the space $F'$, where we use the fibration given in Theorem ~\ref{coneloop}:
	
\[\begin{tikzcd}
	{\Omega(C\Omega X,\Omega X)^{P_l}} & F' & {(C\Omega X,\Omega X})^{\{v_0,v_l\}} & {(C\Omega X,\Omega X)^{P_l}} \\
	{\Omega\ux^{P_l}} & {F} & {\ux^{\{v_0,v_l\}}} & {\ux^{P_l}} \\
	{\prod\limits_{i=0}^l \Omega X_i} & {\prod\limits_{i=1}^{l-1} \Omega X_i} & {X_0 \times X_l} & {\prod\limits_{i=0}^l X_i.}
	\arrow[from=1-3, to=2-3]
	\arrow[from=1-3, to=1-4]
	\arrow[from=1-4, to=2-4]
	\arrow[from=2-3, to=2-4]
	\arrow[from=3-3, to=3-4]
	\arrow[from=2-3, to=3-3]
	\arrow["s",from=2-4, to=3-4]
	\arrow[from=1-2, to=1-3]
	\arrow[from=2-2, to=2-3]
	\arrow[from=3-2, to=3-3]
	\arrow[from=2-2, to=3-2]
	\arrow[from=1-2, to=2-2]
	\arrow["t", from=3-1, to=3-2]
	\arrow[from=2-1, to=2-2]
	\arrow["\Omega s", from=2-1, to=3-1]
	\arrow[from=1-1, to=1-2]
	\arrow[from=1-1, to=2-1]
\end{tikzcd}\] By Theorem ~\ref{coneloop}, the map $\Omega \ux^{P_l} \xrightarrow{\Omega s} \prod_{i=0}^l \Omega X_i$ has a right homotopy inverse. Also, the map $\Omega X_0 \times \Omega X_l \rightarrow \prod_{i=0}^l \Omega X_i$ is an inclusion and so $t$ has a right homotopy inverse. Therefore, the composite $t \circ \Omega s$ has a right homotopy inverse. Hence, it follows from Lemma ~\ref{Fibrethomotopyequiv} that there exists a homotopy equivalence \[F \simeq \prod\limits_{i=1}^{l-1}\Omega X_i \times F'.\]

Now we need to determine $F'$. To do this, we require an alternative homotopy equivalence for $(C\Omega X,\Omega X)^{P_l}$ from \cite{T}. 

\begin{proposition}
\label{gluingconeVertices}
Let $K$ be a simplicial complex on the vertex set $\{1,\cdots,m\}$. Suppose that $K=K_1 \cup \Delta^k$ where: (i) $K_1$ is a simplicial complex on the vertex set $\{1,\cdots,m-1\}$ and $\{i\} \in K_1$ for $1 \leq i \leq m-1$; (ii) $\Delta^k$ is on the vertex set $\{m-k,\cdots,m\}$, and (iii) $K_1 \cap \Delta^k$ is a $(k-1)$-simplex on the vertex set $\{m-k,\cdots,m-1\}$. Then there is a homotopy equivalence \[\cxx^K \simeq \left((\prod\limits_{i=1}^{m-k-1}X_i)*X_m\right)\vee\left(\cxx^{K_1} \rtimes X_m\right).\]
\qedno
\end{proposition}

Applying Proposition ~\ref{gluingconeVertices} with $K_1 = \{v_0\}$ and $\Delta^0 = \{v_l\}$, \begin{equation}\label{twovertdecomp}(C \Omega X,\Omega X)^{\{v_0,v_l\}} \simeq (\Omega X_0 * \Omega X_l) \vee ((C \Omega X,\Omega X)^{\{v_0\}} \rtimes \Omega X_l)\end{equation}\[ \simeq (\Omega X_0 * \Omega X_l) \vee (C \Omega X_0 \rtimes \Omega X_l) \simeq \Omega X_0 * \Omega X_l\] where the last homotopy equivalence follows since $C \Omega X_0$ is contractible and so $C \Omega X_0 \rtimes \Omega X_l$ is contractible. Now applying Proposition ~\ref{gluingconeVertices} with $K_1 = P_{l-1}$ on the vertex set $\{0,\cdots,l-1\}$ and $\Delta^1$ on the vertex set $\{l-1,l\}$ we obtain \[(C\Omega X,\Omega X)^{P_l} \simeq \left((\prod\limits_{i=0}^{l-2} \Omega X) * \Omega X\right) \vee \left((C\Omega X,\Omega X)^{P_{l-1}} \rtimes \Omega X\right).\] The proof of Proposition  ~\ref{gluingconeVertices} uses Lemma ~\ref{coordinatehtpyequiv}, therefore if the hypotheses of Lemma ~\ref{naturalhtpyequiv} also hold, the homotopy equivalence in Proposition ~\ref{gluingconeVertices} is natural. We wish to show this is the case for $(CX,X)^{\{v_0,v_l\}}$ and $(CX,X)^{P_l}$. To do this, we reproduce the proof of Proposition  ~\ref{gluingconeVertices} from \cite{T} and show that the hypotheses of Lemma ~\ref{naturalhtpyequiv} hold. 

\begin{lemma}
\label{naturalendpointinclusion}
Let $P_l$ be the path of length $l$ and let $v_0$, $v_l$ be the end vertices. The homotopy equivalence in Proposition ~\ref{gluingconeVertices} is natural with respect to the inclusion $(CX,X)^{\{v_0,v_l\}}\hookrightarrow (CX,X)^{P_l}$. In particular, this inclusion can be written, up to homotopy equivalences, as the inclusion \[X_0 * X_l \hookrightarrow \left(\left(\prod\limits_{i=0}^{l-2} X_i\right) * X_l\right) \vee \left((C\Omega X,\Omega X)^{P_{l-1}} \rtimes X_l\right)\] into the left wedge summand.
\end{lemma}
\begin{proof}
In the following, we will label the vertices by their corresponding vertices for clarity however we are still considering the case $(CX,X)$. Define $\{v_0,v_l\}$ by the pushout \[\begin{tikzcd}
	\emptyset & {v_l} \\
	{v_0} & {\{v_0,v_l\}.}
	\arrow[from=1-2, to=2-2]
	\arrow[from=1-1, to=1-2]
	\arrow[from=1-1, to=2-1]
	\arrow[from=2-1, to=2-2]
\end{tikzcd}\] By Proposition ~\ref{inducepushout}, this induces a pushout of polyhedral products \[\begin{tikzcd}
	{X_0 \times X_l} & {X_0 \times CX_l} \\
	{CX_0 \times X_l} & {(CX,X)^{\{v_0,v_l\}}.}
	\arrow[from=1-2, to=2-2]
	\arrow["{id_{X_0} \times i_l}", from=1-1, to=1-2]
	\arrow["{i_0 \times id_{X_l}}", from=1-1, to=2-1]
	\arrow[from=2-1, to=2-2]
\end{tikzcd}\] Consider the iterated pushout \begin{equation}\label{verticesnatural}\begin{tikzcd}
	{X_0 \times X_l} & {X_0 \times CX_l} & {X_0} \\
	{CX_0 \times X_l} & {(CX,X)^{\{v_0,v_l\}}} & R
	\arrow["{id_{X_0} \times i_l}", from=1-1, to=1-2]
	\arrow["{i_0 \times id_{X_l}}", from=1-1, to=2-1]
	\arrow[from=1-2, to=2-2]
	\arrow[from=2-1, to=2-2]
	\arrow["{\pi_1}", from=1-2, to=1-3]
	\arrow["g", from=2-2, to=2-3]
	\arrow[from=1-3, to=2-3]
\end{tikzcd}\end{equation} which defines the space $R$ and map $g$, where $\pi_1$ is the projection into the first factor. Observe that the map $i_0 \simeq *$ and so the outer square of (\ref{verticesnatural}) is of the form in Lemma ~\ref{coordinatehtpyequiv}. 

Now define $P_l$ by the pushout \[\begin{tikzcd}
	{x_{l-1}} & {x_{l-1} \cdot x_l} \\
	{P_{l-1}} & {P_l.}
	\arrow[from=2-1, to=2-2]
	\arrow[from=1-2, to=2-2]
	\arrow[from=1-1, to=2-1]
	\arrow[from=1-1, to=1-2]
\end{tikzcd}\] By Proposition ~\ref{inducepushout}, this induces a pushout of polyhedral products \[\begin{tikzcd}
	{\prod\limits_{i=0}^{l-2} X_i \times CX_{l-1} \times X_l} & {\prod\limits_{i=0}^{l-2}X_i \times CX_{l-1} \times CX_{l}} \\
	{(CX,X)^{P_{l-1}} \times X_l} & {(CX,X)^{P_{l}}.}
	\arrow[from=2-1, to=2-2]
	\arrow["{j \times id_{X_l}}", from=1-1, to=2-1]
	\arrow[from=1-1, to=1-2]
	\arrow[from=1-2, to=2-2]
\end{tikzcd}\] Consider the iterated pushout \begin{equation}\label{pathnatural}\begin{tikzcd}
	{\left(\prod\limits_{i=0}^{l-2}X_i \times CX_{l-1}\right) \times X_l} & {\left(\prod\limits_{i=0}^{l-2}X_i \times CX_{l-1}\right) \times CX_l} & {\prod\limits_{i=0}^{l-2}X_i \times CX_{l-1}} \\
	{(CX,X)^{P_{l-1}} \times X_l} & {(CX,X)^{P_l}} & {R'}
	\arrow["{j \times id_{X_l}}", from=1-1, to=2-1]
	\arrow[hook, from=1-1, to=1-2]
	\arrow[from=2-1, to=2-2]
	\arrow[from=1-2, to=2-2]
	\arrow["{\pi_1'}", from=1-2, to=1-3]
	\arrow["g'", from=2-2, to=2-3]
	\arrow[from=1-3, to=2-3]
\end{tikzcd}\end{equation} which defines the space $R'$ and map $g'$, where $\pi_1'$ is the projection into the first factor. The inclusion $j$ factors as the composite \[\prod_{i=0}^{l-2}X_i \times CX_{l-1} \xhookrightarrow{k_0} (CX,X)^{\{v_0,\cdots,v_{l-1}\}} \xhookrightarrow{k_1} (CX,X)^{P_{l-1}}.\] By Lemma ~\ref{conenullhtpy}, $k_0$ is null homotopic, realised by a homotopy \[H': \prod_{i=0}^{l-2}X_i \times CX_{l-1} \times I \rightarrow (CX,X)^{\{v_0,\cdots,v_{l-1}\}}.\] Therefore $j$ is null homotopic, realised by a homotopy $H = k_1 \circ H': \prod_{i=0}^{l-2}X_i \times CX_{l-1} \times I \rightarrow (CX,X)^{P_{l-1}}$. Hence, the outer square of (\ref{pathnatural}) is of the form in Lemma ~\ref{coordinatehtpyequiv}.

By Lemma ~\ref{conenullhtpy}, there is a commutative diagram \begin{equation}\label{homotopycommute}\begin{tikzcd}
	{X_0 \times I} & {CX_0} \\
	{\prod\limits_{i=0}^{l-2} X_i \times CX_{l-1} \times I} & {(CX,X)^{P_{l-1}},}
	\arrow[hook, from=1-2, to=2-2]
	\arrow[hook, from=1-1, to=2-1]
	\arrow["H|_{X_0}", from=1-1, to=1-2]
	\arrow["{H}", from=2-1, to=2-2]
\end{tikzcd}\end{equation} as in the right square of (\ref{eq:pushoutsfornatural}). Therefore, with the outer square of (\ref{verticesnatural}) and (\ref{pathnatural}) for (\ref{eq:pushoutsfornatural}) and (\ref{homotopycommute}) for (\ref{eq:diagramsfornatural}) the hypotheses of Lemma ~\ref{naturalhtpyequiv} are satisfied, and the homotopy equivalence in Lemma ~\ref{coordinatehtpyequiv} is natural for maps $R \rightarrow R'$. In particular, there exists a homotopy commutative diagram \begin{equation}\label{initialnaturaldiag}\begin{tikzcd}
	{R} & {(X_0*X_l) \vee (CX_0 \rtimes X_l)} \\
	{R'} & {\left(\left(\prod\limits_{i=0}^{l-2}X_i \times CX_{l-1}\right)*X_l\right) \vee ((CX,X)^{P_{l-1}} \rtimes X_l).}
	\arrow[from=1-1, to=2-1]
	\arrow["\simeq", from=1-1, to=1-2]
	\arrow[hook, from=1-2, to=2-2]
	\arrow["\simeq", from=2-1, to=2-2]
\end{tikzcd}\end{equation} Since $CX_0$ is naturally homotopy equivalent to a point, so is $CX_0 \rtimes X_l$. Also, $CX_{l-1}$ is naturally homotopy equivalent to a point, and so we can contract $CX_0 \rtimes X_l$ and $CX_{l-1}$ in (\ref{initialnaturaldiag}) to obtain a homotopy commutative diagram \begin{equation}\label{Rhomcomm}\begin{tikzcd}
	{R} & {X_0*X_l} \\
	{R'} & {\left(\left(\prod\limits_{i=0}^{l-2}X_i\right)*X_l\right) \vee ((CX,X)^{P_{l-1}} \rtimes X_l)}
	\arrow[from=1-1, to=2-1]
	\arrow["\simeq", from=1-1, to=1-2]
	\arrow[hook, from=1-2, to=2-2]
	\arrow["\simeq", from=2-1, to=2-2]
\end{tikzcd}\end{equation} where $X_0 * X_l$ is included into the left wedge summand.

The map $R \rightarrow R'$ is a pushout map obtained by mapping the outer square of (\ref{verticesnatural}) to the outer square of (\ref{pathnatural}) via the maps \[X_0 \hookrightarrow \prod\limits_{i=0}^{l-2} X_i \times CX_{l-1}, \text{ } CX_0 \hookrightarrow (CX,X)^{P_{l-1}}, \text{ } X_l \xrightarrow{=} X_l\] where the first two maps are inclusions and the last map is the identity. The uniqueness of strict pushout maps implies that this is the same map as the one obtained by mapping the right square of (\ref{verticesnatural}) to the right square of (\ref{pathnatural}) via \[X_0 \hookrightarrow \prod\limits_{i=0}^{l-2} X_i \times CX_{l-1}, \text{ } (CX,X)^{\{v_0,v_l\}} \hookrightarrow (CX,X)^{P_{l}}, \text{ } CX_l \xrightarrow{=} CX_l\] where the first two maps are inclusions and the last map is the identity. Thus there exists a commutative diagram \begin{equation}\label{PPtoR}\begin{tikzcd}
	{(CX,X)^{\{v_0,v_l\}}} & R \\
	{(CX,X)^{P_l}} & {R'.}
	\arrow[hook, from=1-1, to=2-1]
	\arrow["{g'}", from=2-1, to=2-2]
	\arrow["g", from=1-1, to=1-2]
	\arrow[from=1-2, to=2-2]
\end{tikzcd}\end{equation} Since $\pi_1$ and $\pi_1'$ in the right squares of (\ref{verticesnatural}) and (\ref{pathnatural}) respectively are homotopy equivalences, so are $g$ and $g'$. Therefore, combining (\ref{Rhomcomm}) and (\ref{PPtoR}) we obtain a homotopy commutative diagram \[\begin{tikzcd}
	{(CX,X)^{\{v_0,v_l\}}} & {X_0 *X_l} \\
	{(CX,X)^{P_l}} & {\left(\left(\prod\limits_{i=0}^{l-2} X_i\right) * X_l\right) \vee ((CX,X)^{P_{l-1}} \rtimes X_l).}
	\arrow[hook, from=1-1, to=2-1]
	\arrow[hook, from=1-2, to=2-2]
	\arrow["\simeq", from=1-1, to=1-2]
	\arrow["\simeq", from=2-1, to=2-2]
\end{tikzcd}\]
\end{proof}

Recall $F'$ is the homotopy fibre of the inclusion $(C\Omega X,\Omega X)^{\{v_0,v_l\}} \hookrightarrow (C \Omega X,\Omega X)^{P_l}$. Applying Lemma ~\ref{naturalendpointinclusion}, the map $(C\Omega X,\Omega X)^{\{v_0,v_l\}} \hookrightarrow (C\Omega X,\Omega X)^{P_l}$ can be written up to homotopy equivalences as the inclusion \begin{equation}\label{initialnaturalinclusion} X_0 * X_l \hookrightarrow \left(\left(\prod\limits_{i=0}^{l-2} X_i\right) * X_l\right) \vee \left((C\Omega X,\Omega X)^{P_{l-1}} \rtimes X_l\right).\end{equation} For spaces, $Y_1,\cdots,Y_m$ and $I = \{i_1,\cdots,i_k\} \subseteq [m]$, denote by $\widehat{Y}^I$ the smash product $Y_{i_1} \wedge \cdots \wedge Y_{i_k}$. There is a well-known natural homotopy equivalence \cite{J,Mi} \[\Sigma(Y_1 \times \cdots \times Y_m) \simeq \Sigma\left(\bigvee\limits_{I \subseteq [m]} \widehat{Y}^I\right).\] Using this, there are natural homotopy equivalences \[\left(\prod\limits_{i=0}^{l-2} X_i\right) * X_l \simeq \Sigma \left(\prod\limits_{i=0}^{l-2} X_i\right) \wedge \Omega X_l \simeq \Sigma \left(\bigvee\limits_{I \subseteq [l-2]} \left(\widehat{\Omega X}^I \wedge \Omega X_l\right)\right)\]\[\Sigma \left(\Omega X_0 \wedge \Omega X_l\right) \vee \Sigma\left(\bigvee\limits_{\substack{I \subseteq [l-2] \\ I \neq \{0\}}} \left(\widehat{\Omega X}^I \wedge \Omega X_l\right)\right).\] Therefore, (\ref{initialnaturalinclusion}), can be rewritten up to homotopy equivalences as the inclusion \[\Omega X_0 * \Omega X_l \hookrightarrow (\Omega X_0 * \Omega X_l) \vee C\] where \[C \simeq \Sigma\left(\bigvee\limits_{\substack{I \subseteq [l-2] \\ I \neq \{0\}}} \left(\widehat{\Omega X}^I \wedge \Omega X_l\right)\right) \vee \left((C\Omega X,\Omega X)^{P_{l-1}} \rtimes X_l\right)\] and $(C\Omega X,\Omega X)^{P_{l-1}}$ can be identified by Lemma ~\ref{Pathexplicithtpyequiv}. In particular, since $(C\Omega X,\Omega X)^{P_{l-1}} \in \mathcal{W}$, it follows that $C \in \mathcal{W}$. Therefore, by Proposition ~\ref{inclusionfibre}, \[F' \simeq \Omega\left(C \rtimes \Omega(\Omega X_0 * \Omega X_l)\right).\] Let $Y$ be a pointed topological space. Denote by $Y^{\wedge n}$ the $n$-fold smash product of $Y$. Putting all this together, we obtain the following result.

\begin{theorem}
\label{bookgraphdecomp}
Let $B(l,2l,p)$ be the generalised book graph with $p \geq 2$ and $l \geq 2$. Then there is a homotopy equivalence \[\Omega \ux^{B(l,2l,p)} \simeq \prod\limits_{i=0}^l \Omega X \times \Omega (C\Omega X,\Omega X)^{P_l} \times \Omega \left(\bigvee\limits_{i=1}^p \Sigma F\right)\] where if $l = 2$, $F \simeq \Omega X$ or if $l \geq 3$, $F \simeq \prod_{i=1}^{l-1}\Omega X \times \Omega\left(C \rtimes \Omega(\Omega X * \Omega X)\right)$ and \[C \simeq \Sigma\left(\bigvee_{\substack{I \subseteq [l-2] \\ I \neq \{0\}}} \Omega X^{\wedge |I|+1}\right) \vee \left((C\Omega X,\Omega X)^{P_{l-1}} \rtimes \Omega X\right).\] In particular, $(C\Omega X,\Omega X)^{P_l},\bigvee_{i=1}^p \Sigma F \in \mathcal{W}$.
\end{theorem}
\begin{proof}
We have shown there is a homotopy equivalence \[\Omega \ux^{B(l,2l,p)} \simeq \Omega\ux^{P_l} \times \Omega \left(\bigvee\limits_{i=1}^p \Sigma F\right)\] where if $l=2$, $F \simeq \Omega X$ or if $l \geq 3$, $F \simeq \prod_{i=1}^{l-1}\Omega X \times \Omega\left(C \rtimes \Omega(\Omega X * \Omega X)\right)$ with $C \in \mathcal{W}$. By Proposition ~\ref{coneloop}, \[\Omega\ux^{P_l} \simeq \prod\limits_{i=1}^l \Omega X \times \Omega (C \Omega X,\Omega X)^{P_l}\] where $(C \Omega X,\Omega X)^{P_l} \in \mathcal{W}$ by Lemma ~\ref{Pathexplicithtpyequiv}. 

If $l=2$, clearly $\Sigma F \in \mathcal{W}$. Consider the case $l \geq 3$. There is a homotopy equivalence \[\Sigma F \simeq \Sigma \left(\prod_{i=1}^{l-1}\Omega X\right) \vee \Sigma \left(\Omega\left(C \rtimes \Omega(\Omega X * \Omega X)\right)\right) \vee \Sigma \left(\left(\prod_{i=1}^{l-1}\Omega X\right) \wedge \left(C \rtimes \Omega(\Omega X * \Omega X)\right)\right).\] Iterating the homotopy equivalence $\Sigma (X \times Y) \simeq \Sigma X \vee \Sigma Y \vee \Sigma (X \wedge Y)$ where $X$ and $Y$ are pointed spaces, it follows that $\Sigma \left(\prod_{i=1}^{l-1}\Omega X\right) \in \mathcal{W}$. It suffices to show that $\Sigma \left(\Omega\left(C \rtimes \Omega(\Omega X * \Omega X)\right)\right) \in \mathcal{W}$ since if so, then by shifting the suspension coordinate, it follows that \[\Sigma \left(\left(\prod_{i=1}^{l-1}\Omega X\right) \wedge \left(C \rtimes \Omega(\Omega X * \Omega X)\right)\right) \in \mathcal{W}.\]

Since $C \in \mathcal{W}$, it is a suspension and so $C = \Sigma C'$ where $C'$ is a wedge of smashes of $\Omega X$'s. In particular, there is a homotopy equivalence \[C \rtimes \Omega(\Omega X * \Omega X) \simeq C \vee (\Sigma C' \wedge \Omega(\Omega X * \Omega X)) \simeq C \vee (C' \wedge \Sigma \Omega(\Omega X * \Omega X)).\] By the James construction \cite{J}, $\Sigma \Omega(\Omega X * \Omega X)) \in \mathcal{W}$ and therefore, by shifting the suspension coordinate, it follows that $C \vee (C' \wedge \Sigma \Omega(\Omega X * \Omega X)) \in \mathcal{W}$. Hence, $C \rtimes \Omega(\Omega X * \Omega X) \in \mathcal{W}$. The Hilton-Milnor theorem \cite{Mi} implies that $\Omega (C \rtimes \Omega(\Omega X * \Omega X))$ is a finite type product where each term is a loop of the suspension of smashes of $\Omega X$'s. Iterating $\Sigma (X \times Y) \simeq \Sigma X \vee \Sigma Y \vee \Sigma (X \wedge Y)$ and the James construction, we obtain $\Sigma \left(\Omega\left(C \rtimes \Omega(\Omega X * \Omega X)\right)\right) \in \mathcal{W}$.
\end{proof}

For example, let $X=\mathbb{C}P^{\infty}$. In the case that $X = \mathbb{C}P^\infty$, $\Omega X \simeq S^1$ and so spaces in $\mathcal{W}$ are homotopy equivalent to a wedge of simply connected spheres. The graph $B(l,2l,p)$ is a flag complex, and so by \cite{PT}, it is known that the loop space of the corresponding moment-angle complex is homotopy equivalent to a product of spheres and loops on spheres. This also implies that the loop space of the Davis-Januszkiewicz space decomposes in the same way. However, the technique used does not give any indication as to how many spheres are in the decomposition. We can use Theorem ~\ref{bookgraphdecomp} to give an explicit decomposition. 

\begin{theorem}
\label{Bookgraphdecompdavis}
Let $B(l,2l,p)$ be the generalised book graph with $p \geq 2$ and $l \geq 2$. Then there is a homotopy equivalence \[\Omega DJ_{B(l,2l,p)} \simeq \prod\limits_{i=0}^l S^1 \times \Omega \mathcal{Z}_{P_l} \times  \Omega \left(\bigvee\limits_{i=1}^p \Sigma F\right)\] where if $l = 2$, $F \simeq S^1$ or if $l \geq 3$, $F \simeq \prod_{i=1}^{l-1}S^1 \times \Omega\left(C \rtimes \Omega S^3\right)$ and \[C \simeq \bigvee_{\substack{I \subseteq [l-2] \\ I \neq \{0\}}} S^{|I|+2} \vee \left(\mathcal{Z}_{P_{l-1}} \rtimes S^1\right).\] In particular, $\mathcal{Z}_{P_l},\bigvee_{i=1}^p \Sigma F$ are homotopy equivalent to a wedge of simply connected spheres.
\qedno
\end{theorem}

\subsection*{Data Availability}
Data sharing not applicable to this article as no datasets were generated or analysed during the current study.

\bibliographystyle{amsalpha}

\end{document}